\documentclass[12pt]{amsart}

\usepackage{lineno,hyperref}
\usepackage{amsmath,amsfonts,amsthm}
\usepackage{amssymb,latexsym}
\usepackage{cases}
\usepackage[numbers]{natbib} 
\usepackage{booktabs}
\usepackage{bm} 
\usepackage{cleveref}
\usepackage{multirow}
\usepackage{paralist}
\renewcommand\appendix{\setcounter{secnumdepth}{-2}}

\usepackage{url,cancel}
\usepackage{float} 

\usepackage{caption,enumitem}
\newtheorem{thm}{Theorem}[section]

\newtheorem{lem}[thm]{Lemma}
\newtheorem{prop}[thm]{Proposition}
\newtheorem{cor}[thm]{Corollary}
\theoremstyle{remark}
\newtheorem{re}[thm]{Remark}
\newtheorem{ex}[thm]{Example}
\newtheorem{defn}[thm]{Definition}

\newcommand{\ord}{\mathrm{ord}}

\newcommand{\R}{{\mathbb R}}

\def\rep{{\to\!\!\! -}}

\def\nrep{{\to\!\!\not\! -}}

\makeatletter

\@addtoreset{equation}{section}
\makeatother
\numberwithin{equation}{section}

\begin{document}
\author{Zilong He}
\address{Department of Mathematics, Dongguan University of Technology, Dongguan 523808, China}
\email{zilonghe@connect.hku.hk}
\title[ ]{On classic n-universal quadratic forms over dyadic local fields}
\thanks{ }
\date{\today}
\begin{abstract}
  Let $ n $ be an integer and $ n\ge 2 $. A classic integral quadratic form over local fields is called classic $ n $-universal if it represents all $n$-ary classic integral quadratic forms. We determine the equivalent conditions and minimal testing sets for classic $ n $-universal quadratic forms over dyadic local fields.
\end{abstract}
\maketitle

\section{Introduction} 
The determination problem of classic $n$-universal quadratic forms over $\mathbb{Z}$ has been widely studied (cf. \cite{conway_15-theorem-2000,kim_finiteness_2005,kim_2universal_1999,kim_2universal_1997,ko_on_1937}), since Mordell \cite{mordell_waring_1930} introduced new Waring's problems. 
Classifying $ n $-universal lattices in local fields plays an important role in the global field situation. The characterization for classic $ n $-universal lattices over non-dyadic local fields has been completed in \cite{hhx_indefinite_2021,xu_indefinite_2020} in terms of Jordan splittings. Over general dyadic fields, Beli \cite{beli_universal_2020} completely classified $ 1 $-universal integral lattices by his theory of BONGs (Bases Of Norm Generators) developed in \cite{beli_integral_2003,beli_representations_2006,beli_Anew_2010}, which was recently extended to $ n\ge 2$ by Hu and the author \cite{HeHu2} (also see  \cite{earnes_local_2021,HeHu1,hhx_indefinite_2021,xu_indefinite_2020} for partial results obtained by classical theory). One may be only interested in representations of classic integral lattices. Thus it is natural to ask the necessary and sufficient conditions for classic $ n $-universal lattices in local fields, especially in dyadic fields, because any integral lattice in a non-dyadic local field must be classic integral. The case $ n=1 $ can be solved by Beli's results in \cite{beli_universal_2020} without any barrier, but it is hard to derive general cases from \cite{HeHu2} since the notion of $ n $-universality does not coincide for integral lattices and classic integral lattices when $ n\ge 2$.

In the paper, we give a criterion (Theorem \ref{thm:classicalnuniversaldyadic}) and determine a minimal testing set (Theorem \ref{thm:classicalnuniversaldyadic15theorem}) for classic $ n $-universal lattices for arbitrary $n$ and arbitrary dyadic local fields. As seen in \cite{HeHu2}, the use of BONGs enables us to recognize the patterns from the cases of lower rank and predict the lattice sets for testing $ n $-universality. In view of that, our main results will be shown by the techniques used there and be formulated in the language of BONGs for compactness. Although the treatment here is similar to \cite{HeHu2}\footnote{As pointed out in  \cite[Section 4]{beli_universal_2020}, one may simplify the $ n $-universal problem to some extent by reducing to the cases $ n=1,2,3,4 $.}, it will be seen that the classic integral case is more complicated, for example, finding maximal lattices is not sufficient for determining minimal testing sets.

All quadratic spaces and lattices will be assumed to be non-degenerate. Let $F$ be an algebraic number field or a local field, and $ \mathcal{O}_{F} $ the ring of integers of $ F $. Let $ V $ be a quadratic space over $ F $ associated with the symmetric bilinear form $ B:V\times V\to F $ and put $ Q(x):=B(x,x) $ for any $ x\in V $.  For a subset $ L$ of $ V $, we say that $ L $ is an \textit{$\mathcal{O}_{F} $-lattice} in $ V $ if it is a finitely generated $ \mathcal{O}_{F} $-module and say that it is on $ V $ if $ V=FL $, i.e. $ V $ is the space spanned by $ L $ over $ F $. We call an $ \mathcal{O}_{F} $-lattice $ L $ \textit{integral} if $ \mathfrak{n}L\subseteq \mathcal{O}_{F} $ and \textit{classic integral} if $  \mathfrak{s}L\subseteq \mathcal{O}_{F}$, respectively, where $ \mathfrak{n}L $ and $ \mathfrak{s}L $ denote the norm and the scale of $ L $, respectively (cf. \cite[\S 82E, p.\hskip 0.1cm 227]{omeara_quadratic_1963}). Following \cite[\S 82A, p.\hskip 0.1cm 220]{omeara_quadratic_1963}, we write $ N\rep L $ when a lattice $ N $ is represented by another lattice $ L $. Similarly for quadratic spaces. 

Let $ x_{1},\ldots, x_{m} $ be pairwise orthogonal vectors of $ V  $ with $ Q(x_{i})=a_{i} $. Then we write $ V\cong [a_{1},\ldots,a_{m}] $ if $ V=Fx_{1}\perp \ldots \perp Fx_{m} $ and $ L\cong \langle a_{1},\ldots,a_{m}\rangle $ if $ L=\mathcal{O}_{F}x_{1}\perp \ldots\perp \mathcal{O}_{F}x_{m} $.  When $F$ is a dyadic local field, we also write $ L\cong \prec a_{1},\ldots, a_{m} \succ $ if $ x_{1},\ldots, x_{m} $ is a BONG for $ L $ (cf. Section \ref{sec:pre}).

Unless otherwise stated, we assume that $ F $ is a dyadic local field. Let $ \mathcal{O}_{F}^{\times}$ be the group of units, $ \mathfrak{p} $ the prime ideal of $ F $ and $ \pi\in \mathfrak{p} $ a fixed prime. For $ c\in F^{\times}:=F\backslash \{0\} $, let $ c=\mu\pi^{k} $, where $ \mu\in \mathcal{O}_{F}^{\times} $ and $ k\in \mathbb{Z} $. We define the \textit{order of $ c $} to be $ \ord (c):=k $ and formally set $ \ord (0)=\infty $, and then put $ e:=\ord (2)$. For a fractional or zero ideal $ \mathfrak{c} $ of $ F $, put $ \ord (\mathfrak{c}):=\min\{\ord (a)\mid a\in \mathfrak{c}\}$. We also define the \textit{quadratic defect of $ c $} by $ \mathfrak{d}(c)=\bigcap_{x\in F}(c-x^{2})\mathcal{O}_{F}$ and the \textit{order of relative quadratic defect} by $ d:F^{\times}/F^{\times 2}\rep \mathbb{N}\cup \{\infty\} $, $ d(c)=\ord (c^{-1}\mathfrak{d}(c) )$. As usual, let $ \Delta:=1-4\rho $ be a fixed unit with $ \mathfrak{d}(\Delta)=4\mathcal{O}_{F} $ and let $N\mathfrak{p}$ be the number of elements in the residue field of $F$. As in \cite{HeHu2}, let $ \mathcal{U} $ be a complete system of representatives of $ \mathcal{O}_{F}^{\times }/\mathcal{O}_{F}^{\times 2} $ such that $ d(\delta)=\ord(\delta-1) $ for all $ \delta\in\mathcal{U} $, and write $\mathcal{U}_{1}=\{\delta\in \mathcal{U}\mid d(\delta)=1\}$.

Let $ \gamma\in F^{\times} $ and $ \xi,\eta\in F $. We write $ \gamma A(\xi,\eta)$ for the binary lattice whose Gram matrix is $ \gamma\begin{pmatrix}
	\xi  &  1  \\
	1  & \eta
\end{pmatrix} $. For $ s,t\in \mathbb{N} $, we denote by 
$ \mathbf{H}_{t} $ the $ \mathcal{O}_{F} $-lattice $ \prec \pi^{t},-\pi^{-t}\succ $ with $ 0\le t\le e $ and let $ \mathbf{H}_{t}^{s}$ stand for the orthogonal sum of $s $ copies of $ \mathbf{H}_{t} $. For $ h,k\in\mathbb{Z} $, we write $ [h,k]^{E} $ (resp. $ [h,k]^{O} $) for the set of all even (resp. odd) integers $ i $ with $ h\le i\le k  $.

 As in \cite{hhx_indefinite_2021}, we call a quadratic space $ V $ over $ F $ \textit{$ n $-universal} if it represents all quadratic spaces of dimension $ n $ over $ F $, and call a classic integral $ \mathcal{O}_{F} $-lattice $  L $ \textit{classic $ n $-universal} if it represents all classic integral $ \mathcal{O}_{F} $-lattices of rank $ n $. 

\medskip

The following theorem characterizes classic $ n $-universal $ \mathcal{O}_{F} $-lattices in terms of good BONGs for $ n\ge 2 $.
\begin{thm}\label{thm:classicalnuniversaldyadic}
	Let $ n $ be an integer and $ n\ge 2 $. Let $M\cong \prec a_{1},\ldots,a_{m}\succ$ be a classic integral $ \mathcal{O}_{F} $-lattice relative to some good BONG and $ R_{i}=\ord\, (a_{i}) $ for $ 1\le i\le m $. Then $ M $ is classic $n$-universal if and only if $ m\ge n+3 $ and the following conditions hold.
    \begin{enumerate}
    	\item[\rm (i)] $ R_{i}=0 $ for $ 1\le i\le n $.
    	
    	\item[\rm (ii)] $ n $ is even, $ R_{n+1}=0 $ and the following conditions hold.  	
    	\begin{enumerate}
    		\item[\rm (1)] $ R_{n+2}\in \{0,1\} $; and if $ R_{n+2}=0 $, then the following conditions hold.  		
    		\begin{enumerate}
    			\item[\rm (a)] $ d((-1)^{\frac{n+2}{2}}a_{1}\cdots a_{n+2})=1 $ or $ R_{n+3}\in \{0,1\} $.
    			
    			\item[\rm (b)] If $ e>1 $, $ R_{n+2}=R_{n+3}=0 $ and $ d((-1)^{\frac{n+2}{2}}a_{1}\cdots a_{n+2})>1 $, then $ d(-a_{j}a_{j+1})=1-R_{j+1} $ for some $ 1\le j\le m-1 $.
    		\end{enumerate}
  		
    		\item[\rm (2)] $ R_{n+3}-R_{n+2}\le 2e $.
    	\end{enumerate}
    
    \item[\rm (iii)] $ n $ is odd and the following conditions hold.    
    \begin{enumerate}
    	\item[\rm (1)] $ R_{n+1}\in \{0,1\} $; and if $ R_{n+1}=0$, then the following conditions hold.
    	
    	\begin{enumerate}
    		\item[\rm (a)]  $ d((-1)^{\frac{n+1}{2}}a_{1}\cdots a_{n+1})=1 $ or $ R_{n+2}\in \{0,1\} $.
    		
    		\item[\rm (b)] If $ e>1 $, $ R_{n+1}=R_{n+2}=0 $ and $ d((-1)^{\frac{n+1}{2}}a_{1}\cdots a_{n+1})>1 $, then
    		
    		 $ d(-a_{j}a_{j+1})=1-R_{j+1} $ for some $ 1\le j\le  m-1  $.
    	\end{enumerate}
    	
    	\item[\rm (2)] Suppose either $ R_{n+1}=1 $ or $ R_{n+2}>1 $.     	
    	\begin{enumerate}
    		\item[\rm (a)] If $ R_{n+2}-R_{n+1} $ is even, then $ R_{n+3}+R_{n+2}-2R_{n+1}\le 2e-2 $ or
    		
    			$ d(-a_{j}a_{j+1})\le 2e+R_{n+1}-R_{j+1}-1 $ for some $ n+2\le j\le m-1$.
    		
    		\item[\rm (b)] If $ R_{n+2}-R_{n+1} $ is odd, then $ R_{n+3}+R_{n+2}-2R_{n+1}\le 2e $ or 
    		
    		$ d(-a_{j}a_{j+1})\le 2e+R_{n+1}-R_{j+1} $ for some $ n+2\le j\le m-1$.
    	\end{enumerate}
    	
    	\item[\rm (3)] $ R_{n+3}-R_{n+2}\le 2e $.   
    \end{enumerate} 
    \end{enumerate}
\end{thm}	
\begin{re}
	\begin{enumerate}
		\item[\rm (i)]  Based on the terminology introduced by Beli, Theorem \ref{thm:classicalnuniversaldyadic} can be stated more compactly (see Theorems \ref{thm:classicaleven-nuniversaldyadic} and \ref{thm:classicalodd-nuniversaldyadic}).
		
		\item[\rm (ii)] From the theorem we immediately see that there is no quaternary classic $ 2 $-universal lattice over any dyadic local field, which recovers \cite[Proposition 4.6]{hhx_indefinite_2021}.
	\end{enumerate} 
\end{re}
For $n\ge 2$, we also determine a minimal testing set for classic $ n $-universality.
\begin{thm}\label{thm:classicalnuniversaldyadic15theorem}
	Let $ n$ be an integer and $ n\ge 2 $. Let $ M $ be a classic integral $ \mathcal{O}_{F} $-lattice. 
	\begin{enumerate}
		\item[\rm(i)] When $ n $ is even, $ M $ is classic $ n $-universal if and only if it represents the following \\$ 8(N\mathfrak{p})^{e}-4 (N\mathfrak{p})^{e-1}+1+u_{e} $ classic integral $ \mathcal{O}_{F} $-lattices
		\begin{align*}
			&\mathbf{H}_{e}^{\frac{n}{2}}, \quad\mathbf{H}_{1}^{\frac{n-2}{2}}\perp A(2,2\rho)\quad\text{(if $ e=1 $)},\\ 
			&\mathbf{H}_{0}^{\frac{n-2}{2}}\perp \langle 1,-\varepsilon\pi \rangle, \quad  \mathbf{H}_{0}^{\frac{n-2}{2}}\perp  \langle \Delta,-\Delta\varepsilon\pi \rangle,  \\
			&\mathbf{H}_{0}^{\frac{n-2}{2}}\perp  A(1,-(\delta-1)),  \quad \mathbf{H}_{0}^{\frac{n-2}{2}}\perp  (1+4\rho (\delta-1)^{-1})A(1,-(\delta-1)) 
		\end{align*}
		for all $ \varepsilon\in \mathcal{U} $ and for all $ \delta\in \mathcal{U}_{1} $, where $ u_{e}=1 $ if $ e=1 $, and $ u_{e}=0 $, otherwise.
		
		\item[\rm(ii)] When $ n $ is odd, $ M $ is classic $n$-universal if and only if  it represents the following $ 8(N\mathfrak{p})^{e} $ classic integral $ \mathcal{O}_{F} $-lattices
		\begin{align*}
			&\mathbf{H}_{0}^{\frac{n-1}{2}}\perp  \langle \varepsilon \rangle, \quad \mathbf{H}_{0}^{\frac{n-3}{2}}\perp \varepsilon(1+4\rho\pi^{-1}) A(1,-\pi)\perp \langle \varepsilon(1+\pi) \rangle,\\  
			&\mathbf{H}_{0}^{\frac{n-1}{2}}\perp    \langle \varepsilon\pi \rangle, \quad\mathbf{H}_{0}^{\frac{n-3}{2}}\perp  A(1,4\rho)\perp \langle \Delta\varepsilon\pi \rangle
		\end{align*}
		for all $  \varepsilon\in \mathcal{U} $.
		
		\item[\rm(iii)] The set of lattices listed in (i) (resp. (ii)) is minimal (in the sense of \cite[p.\hskip 0.1cm 4]{HeHu2}). 
	\end{enumerate}
\end{thm}
\begin{re}
	By \cite[Corollary 2.3(ii)]{HeHu2}, we have
	 \begin{align*}
		 	\mathbf{H}_{t}=\prec \pi^{t},-\pi^{-t}\succ\cong
		 	\begin{cases}
			 		A(\pi^{t},0)  &\text{if $ 0\le t<e $}, \\
			 		A(2,0)\cong A(0,0) &\text{if $   t=e $}.
			 	\end{cases}
		 \end{align*}
\end{re}
We give several applications of Theorems \ref{thm:classicalnuniversaldyadic} and \ref{thm:classicalnuniversaldyadic15theorem} in global representations. In the remainder of this section, we let $K$ be an algebraic number field. A classic integral $\mathcal{O}_{K}$-lattice $L$ is called \textit{classic $n$-universal} over $\mathcal{O}_{K}$ (cf. \cite{hhx_indefinite_2021,kim_finiteness_2005}) if one of the following statements holds:
 \begin{enumerate}
 	\item[\rm (i)] $L$ is positive definite and represents all positive definite classic integral $\mathcal{O}_{K}$-lattices of rank $n$; 
 	
 	\item[\rm (ii)]  $L$ is indefinite and represents all classic integral $\mathcal{O}_{K}$-lattices of rank $n$.
 \end{enumerate} 
 
Although the results in \cite{hsia_positive_1978} indicate that universal lattices always exist in a totally real number field, not all number fields admit an $ n $-universal lattice for some specific values of the rank (see  \cite{chan_ternary-1996,kala_2021,sasaki_2universal_2006,sasaki_quaternary_2009} for recent developments). 

The following theorem provides a necessary condition for an algebraic number field which admits some classic $ n $-universal lattice.
\begin{thm}\label{thm:necessarycon}
	 Let $\mathfrak{p}$ be a dyadic prime of $K$. Let $L$ be a classic integral $ \mathcal{O}_{K} $-lattices of rank $ m\ge n+3\ge 4 $. If $ L_{\mathfrak{p}}\cong \prec a_{1}(\mathfrak{p}),\ldots, a_{m}(\mathfrak{p}) \succ  $ is classic $n$-universal and $ d_{\mathfrak{p}}(a_{i}(\mathfrak{p})a_{i+1}(\mathfrak{p}))>1 $ for $1\le i\le m-1$, then $e_{\mathfrak{p}}=1$.
	
	Thus if the above conditions hold for each dyadic prime of $K$, then $ 2 $ is unramified in $ K $, i.e., the discriminant of $ K $ is odd.
\end{thm}
\begin{ex}
	Let $ K_{\ell}=\mathbb{Q}(\zeta_{\ell}+\zeta_{\ell}^{-1}) $ be the maximal real subfield of the $ \ell^{th}$ cyclotomic field, where $ \zeta_{\ell} $ is the primitive $ \ell^{th} $ root of unity.
	Suppose $ \ell=2^{h}\ge 8 $. Then the discriminant of $ K $ is $  2^{(h-1)2^{h-2}-1} $ by \cite[Theorem 1]{liang_on_1976}. Hence $ 2 $ totally ramifies in $ K_{\ell} $. Also, $ 2\mathcal{O}_{K_{\ell}}=\mathfrak{p}^{2^{h-2}} $ for some prime ideal $ \mathfrak{p} $ and $e_{\mathfrak{p}}=2^{h-2}>1$.
	
	Consider the rank $ m $ classic integral $ \mathcal{O}_{K_{\ell}} $-lattice $ L\cong \langle 1,\ldots, 1,3\rangle $ in the orthogonal base $ \{x_{1},\ldots,x_{m}\} $ of $K_{\ell}L$. Then $ L_{\mathfrak{p}}\cong \prec 1,\cdots, 1,3\succ $ relative to the good BONG $\{(x_{1})_{\mathfrak{p}},\ldots,(x_{m})_{\mathfrak{p}} \} $ by Lemma \ref{lem:goodBONGequivcon}. Since $ 3=1+2 $, we have $  d_{\mathfrak{p}}(3)\ge e_{\mathfrak{p}}>1 $. Hence for $ m\ge n+3 $, $ L_{\mathfrak{p}} $ is not classic $ n $-universal by Theorem \ref{thm:necessarycon}, so $L$ is not classic $n$-universal over $ \mathcal{O}_{K_{\ell}} $ by Proposition \ref{prop:Lp-nuniversal}. 
\end{ex}
Recall from \cite{kim_2universal_1997} that there are exactly five positive definite quinary diagonal $\mathbb{Z}$-lattices that are classic $2$-universal over $\mathbb{Z}$. However, we have the following result for general $\mathcal{O}_{K}$.
\begin{thm}\label{thm:classic-n-universal-rankn+3}
	If the discriminant of $K$ is even, then there is no diagonal classic integral $\mathbb{Z}$-lattice of rank $n+3$ that is classic $n$-universal over $\mathcal{O}_{K}$.
\end{thm}
All positive definite quinary non-diagonal classic $2$-universal $\mathbb{Z}$-lattices were further determined in \cite{kim_2universal_1999}, but they are also not classic $2$-universal over $\mathcal{O}_{K}$ when the discriminant of $K$ is even. Indeed, one can take $n=2$, $K=\mathbb{Q}$, and $E$ to be a totally real number field with even discriminant in the following theorem.
\begin{thm}\label{thm:classic-n-universal-rankn+3-ramified}
 Let $E/K$ be a finite extension. Suppose that there exist dyadic primes $\mathfrak{P}|\mathfrak{p}$ such that $E_{\mathfrak{P}}/K_{\mathfrak{p}}$ is ramified. If a classic integral $\mathcal{O}_{K}$-lattice of rank $n+3$ is classic $n$-universal over $\mathcal{O}_{K}$, then it is not classic $n$-universal over  $\mathcal{O}_{E}$.
\end{thm}
Write $ I_{m} $ for the sum of $ m $ squares. We also generalize Siegel's theorem on sums of squares \cite[Theorem II]{siegel_sums_1945}.
\begin{thm}\label{thm:siegelthmII}
Suppose that $K$ is not totally real. If $ m\ge n+3\ge 4 $, then $I_{m}$ represents all positive definite classic integral $\mathcal{O}_{K}$-lattices of rank $n$ if and only if the discriminant of $ K $ is odd.
\end{thm}
The remaining sections are organized as follows. In Section \ref{sec:pre}, we introduce Beli's representation theory and a testing set for classic $n$-universality. In Section \ref{sec:sufficient conditions}, we investigate the conditions (i)-(iv) of Theorem \ref{thm:beligeneral}, which will be used in subsequent sections. In Sections \ref{sec:classicaleven-con} and \ref{sec:classicalodd-con}, we determine the equivalent conditions for classic $ n $-universal lattices for $ n\ge 2 $. Then we prove Theorems \ref{thm:classicalnuniversaldyadic} and \ref{thm:necessarycon} in Section \ref{sec:proof-main}, and Theorem \ref{thm:classicalnuniversaldyadic15theorem} in Section \ref{sec:classicalcstheorem-mini} and Theorems \ref{thm:classic-n-universal-rankn+3}, \ref{thm:classic-n-universal-rankn+3-ramified} and \ref{thm:siegelthmII} in Section \ref{sec:app}. 

\section{Preliminaries }\label{sec:pre}
  Any unexplained notation or definition can be found in \cite{HeHu2} or a series of papers by Beli \cite{beli_thesis_2001,beli_integral_2003,beli_representations_2006,beli_Anew_2010,beli_representations_2019,beli_universal_2020}.
\begin{lem}\label{lem:goodBONGequivcon}\citep[Lemma 2.2]{HeHu2}
	Let $ x_{1},\ldots, x_{m}  $ be pairwise orthogonal vectors in $ V $ with $ Q(x_{i})=a_{i} $ and $ R_{i}=\ord (a_{i}) $.  Then $ x_{1},\ldots,x_{m} $ forms a good BONG for some lattice in $V$ if and only if the following conditions are satisfied:  
	\begin{align}\label{eq:GoodBONGs}
		R_{i}\le R_{i+2} \quad \text{for all $ 1\le i\le m-2$,}
	\end{align}
	and  
	\begin{align}\label{eq:BONGs}
		R_{i+1}-R_{i}+d(-a_{i}a_{i+1})\ge 0 \quad\text{and}\quad  R_{i+1}-R_{i}\ge -2e \quad \text{for all $ 1\le i\le m-1 $}. 
	\end{align}
\end{lem}
Let $ M\cong \prec a_{1},\ldots, a_{m} \succ $ be an $ \mathcal{O}_{F} $-lattice  relative to some good BONG. Define the \textit{$R$-invariants} by $ R_{i}=R_{i}(M):=\ord(a_{i})$, where $1\le i\le m$, and the \textit{$\alpha$-invariants} by $\alpha_{i}=\alpha_{i}(M):=\min\{T_{0}^{(i)},\ldots, T_{m-1}^{(i)}\} $, where 
 \begin{align}\label{T}
 	T_{0}^{(i)}=\dfrac{R_{i+1}-R_{i}}{2}+e, \quad T_{j}^{(i)}=
 	\begin{cases}
 		R_{i+1}-R_{j}+d(-a_{j}a_{j+1}) &\text{if $ 1\le j\le i $}\,, \\
 		R_{j+1}-R_{i}+d(-a_{j}a_{j+1})  &\text{if $ i\le j\le m-1 $}.
 	\end{cases}
 \end{align}
  A concise formula is given in \cite[Corollary 2.5(i)]{beli_Anew_2010} for $ \alpha_{i} $:
 \begin{align}\label{eq:alpha-defn}
 	\alpha_{i}=\min\{\dfrac{R_{i+1}-R_{i}}{2}+e,R_{i+1}-R_{i}+d[-a_{i}a_{i+1}]\},
 \end{align}
 where $d[-a_{i}a_{i+1}]:=\min\{d(-a_{i}a_{i+1}),\alpha_{i-1},\alpha_{i+1}\}$.
 
Recall from \cite[Corollary 4.4(iv)]{beli_integral_2003} that $
	\ord\,(\mathfrak{s}M)= \min\{R_{1},(R_{1}+R_{2})/2\}$.
Also, if $ R_{1}+R_{2} $ is odd, then \cite[Corollary 2.3(i)]{HeHu2} implies $ R_{2}-R_{1}> 0$, i.e. $R_{1}<(R_{1}+R_{2})/2$. So the minimum is $R_{1}$. Hence $M$ is classic integral if and only if $ \mathfrak{s}M\subseteq \mathcal{O}_{F} $, which is equivalent to  
\begin{align}
	&R_{1}\ge 0\quad\text{and} \label{eq:integralcondition} \\
		 & R_{1}+R_{2} \ge 0. \label{eq:classicallyintegral}
\end{align} 
We also have
\begin{align}\label{eq:classicallyintegralR2}
	R_{2} \ge \max\{-\dfrac{d(-a_{1}a_{2})}{2},-e\}
\end{align}
by \eqref{eq:BONGs} and \eqref{eq:classicallyintegral}.

We collect some results for the invariants $ R_{i} $ and $ \alpha_{i} $. 
\begin{prop}\label{prop:Rproperty}
	Suppose $ 1\le i\le m-1 $. 
	\begin{enumerate}
		\item[\rm (i)] $ R_{i+1}-R_{i}>2e$ (resp. $ =2e $, $ <2e $) if and only if $ \alpha_{i}>2e $ (resp. $ =2e $, $ <2e $).
		
		\item[\rm (ii)] If $ R_{i+1}-R_{i}\ge 2e$ or $ R_{i+1}-R_{i}\in \{-2e,2-2e,2e-2\} $, then $ \alpha_{i}=(R_{i+1}-R_{i})/2+e $.
		
		\item[\rm (iii)]  If $ R_{i+1}-R_{i}\le 2e $, then $ \alpha_{i}\ge R_{i+1}-R_{i} $. Also, the equality holds if and only if $ R_{i+1}-R_{i}=2e $ or $ R_{i+1}-R_{i} $ is odd.
		
		\item[\rm (iv)] If $ R_{i+1}-R_{i} $ is odd, then $ \alpha_{i}=\min\{(R_{i+1}-R_{i})/2+e,R_{i+1}-R_{i}\} $ and hence $ R_{i+1}-R_{i}> 0 $.
		
		\item[\rm (v)] $ R_{i}+\alpha_{i} $ is increasing and $ -R_{i+1}+\alpha_{i} $ is decreasing.   
		
		\item[\rm (vi)] If $ R_{i}+R_{i+1}=R_{j}+R_{j+1} $ for some $ j $ with $ i\le j\le m-1 $, then $ R_{i}+\alpha_{i}=\ldots=R_{j}+\alpha_{j} $.
	\end{enumerate}  
\end{prop}
\begin{proof}
	See \cite[Corollaries 2.8(i), 2.9(i), Lemma 2.7(iii), Corollary 2.9(ii), Lemma 2.2, Corollary 2.3(i)]{beli_Anew_2010} for (i)-(vi).
\end{proof} 
\begin{prop}\label{prop:alphaproperty}
		Suppose $ 1\le i\le m-1 $.
		\begin{enumerate}
			\item[\rm (i)]  Either $ 0\le \alpha_{i}\le 2e $ and $ \alpha_{i}\in \mathbb{Z} $, or $ 2e<\alpha_{i}<\infty $ and $ 2\alpha_{i}\in \mathbb{Z} $; thus $ \alpha_{i}\ge 0 $.
			
			\item[\rm (ii)] $ \alpha_{i}=0 $ if and only if $ R_{i+1}-R_{i}=-2e $.
				
			\item[\rm (iii)] $ \alpha_{i}=1 $ if and only if either $ R_{i+1}-R_{i}\in \{2-2e,1\} $, or $ R_{i+1}-R_{i}\in [4-2e,0]^{E} $ and $ d[-a_{i}a_{i+1}]=R_{i}-R_{i+1}+1 $.
			
			\item[] In particular, if $e=1$, then $\alpha_{i}=1$ if and only if $R_{i+1}-R_{i}\in \{0,1\}$.
			
			\item[\rm (iv)] If $ \alpha_{i}=0 $, i.e., $ R_{i+1}-R_{i}=-2e $, then $   d[-a_{i}a_{i+1}]\ge 2e $.
			\item[\rm (v)] If $ \alpha_{i}=1 $, then $ d[-a_{i}a_{i+1}]\ge R_{i}-R_{i+1}+1 $. Also, the equality holds if $ R_{i+1}-R_{i}\not=2-2e $. 
			
			\item[\rm (vi)] If $R_{i+1}-R_{i}+d(-a_{i}a_{i+1})=1$, then $\alpha_{i}=1$.
		\end{enumerate}
\end{prop}
 \begin{proof}
 	See \cite[Corollary 2.8(iii), Lemma 2.7(i)]{beli_Anew_2010} for (i) and (ii); see \cite[Lemma 2.8]{beli_Anew_2010} for (iii)-(v). For (vi), since $R_{i+1}-R_{i}=1-d(-a_{i}a_{i+1})\not=-2e$, we have $\alpha_{i}\ge 1$ by (ii). Also, $ \alpha_{i}\le R_{i+1}-R_{i}+d(-a_{i}a_{i+1})=1$ by \eqref{eq:alpha-defn}. Thus $\alpha_{i}=1$.
 \end{proof}
 \begin{prop}\label{prop:Ralphaproperty3}
	Suppose that $M$ is classic integral.
	\begin{enumerate}
		\item[\rm (i)] We have $ R_{j}\ge R_{i}\ge 0 $ for all  $ i,j\in [1,\,m]^O$ with $ j\ge i $ and $ R_{j}\ge R_{i}\ge -e $ for all $ i,j \in [1,\,m]^E$ with $ j\ge i $. Also, $R_{i}+R_{i+1}\ge 0$ for $1\le i\le m-1$.
	
		\item[\rm (ii)]	If $ R_{j}=0 $ for some $ j\in [1,m]^{O} $, then $ R_{i}=0 $ for all  $ i\in [1,j]^{O} $ and $ R_{i} $ is even for all $ 1\le i\le j $. 
		
		\item[\rm (iii)] If $ R_{j}=-e$ for some $ j\in [1,m]^{E} $, then for each $ i\in [1,j]^{E} $, we have $ R_{i-1}=e $, $ R_{i}=-e $ and $ d(-a_{i-1}a_{i})\ge d[-a_{i-1}a_{i}]\ge 2e $. Consequently, $ d[(-1)^{j/2}a_{1,j}]\ge 2e $.
				
		\item[\rm (iv)] If $R_{1}=0$, then $R_{i}$ is nonnegative for $1\le i\le m$; if moreover $R_{j}+R_{j+1}\le 0$ for some $1\le j\le m-1$, then $R_{i}=0$ for $1\le i\le j$. 
		
		\item[\rm (v)] If $R_{j}=R_{j+1}=0$ for some $1\le j\le m-1$, then $R_{i}=0$ for $1\le i\le j+1$.
		
		\item[\rm (vi)] Suppose that $1\le j\le m-1$ and $R_{i}=0$ for $1\le i\le j$. If $\alpha_{k}\le 1$ for some $1\le k\le j$, then $\alpha_{i}=1$ for $1\le i\le j-1$.
	\end{enumerate}
\end{prop}
\begin{proof}
	(i) See \cite[Proposition 2.7(i)]{HeHu2} for the first statement. For the second one, we have $R_{i}+R_{i+1}\ge R_{1}+R_{2}\ge 0$ from the first statement and \eqref{eq:classicallyintegral}.
	
	(ii) See \cite[Proposition 2.7(ii)]{HeHu2}.
	
	(iii) For each $i\in [1,j]^{E}$, by (i) and \eqref{eq:classicallyintegral}, we have $R_{i-1}+R_{j}\ge R_{1}+R_{2}\ge 0$ and hence $R_{i-1}\ge -R_{j}=e$. On the other hand, by  (i) and \eqref{eq:BONGs}, we also have $R_{i-1}\le R_{j-1}\le R_{j}+2e=e$. Hence $R_{i-1}=e$. So $-e\le R_{i}\le -R_{i-1}=-e$ by (i), i.e. $R_{i}=-e$. 
	    
	By Proposition \ref{prop:alphaproperty}(iv), we have $d(-a_{i-1}a_{i})\ge d[-a_{i-1}a_{i}]\ge 2e$, so $ d[(-1)^{j/2}a_{1,j}]\ge 2e $ by the domination principle.
	 
	(iv) Since $R_{1}=0$, it follows that $R_{2}\ge 0$ from \eqref{eq:classicallyintegral}. By the first part of (i), all $R_{i}$ are nonnegative. Hence if also $R_{j}+R_{j+1}\le 0$, then $ R_{i}=0$ for $1\le i\le j+1$.
	
	(v)	Since $R_{j}=R_{j+1}=0$ with odd $j$ or $j+1$, we have $R_{1}=0$ by (ii). Since also $R_{j}+R_{j+1}=0$, we obtain $R_{i}=0$ for $1\le i\le j$ by (iv).
	
	 (vi) Since $R_{1}=R_{2}=0$, Propositions \ref{prop:alphaproperty}(ii) implies $\alpha_{1}\ge 1$. If $\alpha_{k}\le 1$, then 
	  \begin{align*}
	  	1\le \alpha_{1}=\alpha_{1}+R_{1}\le \alpha_{k}+R_{k}=\alpha_{k}\le 1
	  \end{align*}
	   by Proposition \ref{prop:Rproperty}(v) and thus $\alpha_{1}=\alpha_{k}=1$. If $k\le j-1$, then we are done by Proposition \ref{prop:Rproperty}(vi). If $k=j$, for $1\le i\le j-1$, then
	\begin{align*}
	1=\alpha_{1}=\alpha_{1}+R_{1}\le \alpha_{i}+R_{i}\le  \alpha_{j}+R_{j}=\alpha_{k}+R_{k}=\alpha_{k}=1
	\end{align*}
	by Proposition \ref{prop:Rproperty}(v). This implies $\alpha_{i}=\alpha_{i}+R_{i}=1$ for $1\le i\le j-1 $.
\end{proof}

Let $ N\cong \prec b_{1},\cdots, b_{n}\succ $ be another $ \mathcal{O}_{F} $-lattice relative to some good BONG and $ n\le m $. Write  $ S_{i}=R_{i}(N) $ and $ \beta_{i}=\alpha_{i}(N) $.
For $ 0\le i,j\le m $, we denote by
\begin{align*}
d[ca_{1,i}b_{1,j}]=\min\{d(ca_{1,i}b_{1,j}),\alpha_{i},\beta_{j}\}\quad c\in F^{\times},
\end{align*}
where $ \alpha_{i} $ (resp. $\beta_{j}$) is ignored if $ i\in \{0,m\} $ (resp. if $ j\in\{0,n\} $). For any $ 1\le i\le \min\{m-1,n\} $, we define 
\begin{multline*}
A_{i}=A_{i}(M,N):=\min\{\dfrac{R_{i+1}-S_{i}}{2}+e,R_{i+1}-S_{i}+d[-a_{1,i+1}b_{1,i-1}],\\R_{i+1}+R_{i+2}-S_{i-1}-S_{i}+d[a_{1,i+2}b_{1,i-2}]\},
\end{multline*}
where the term $ R_{i+1}+R_{i+2}-S_{i-1}-S_{i}+d[a_{1,i+2}b_{1,i-2}] $ is ignored if $ i\in \{1,m-1\} $.
If $ n\le m-2 $, we define
\begin{align*}
S_{n+1}+A_{n+1}:=\min\{R_{n+2}+d[-a_{1,n+2}b_{1,n}],R_{n+2}+R_{n+3}-S_{n}+d[a_{1,n+3}b_{1,n-1}]\},
\end{align*}
where the term $R_{n+2}+R_{n+3}-S_{n}+d[a_{1,n+3}b_{1,n-1}]$ is ignored if $ n=m-2 $.

Now, we are ready to formulate the representation theorem for two lattices by Beli (cf. \cite[Theorem  4.5]{beli_representations_2006} and \cite[Theorem 2.8]{HeHu2}).
\begin{thm}\label{thm:beligeneral}
	Suppose $  n\le m$. Then $ N\rep M $ if and only if  $ FN\rep FM $ and the following conditions hold:
	\begin{enumerate}
		\item[\rm (i)] For any $ 1\le i\le n $, we have either $ R_{i}\le S_{i} $, or $ 1<i<m $ and $ R_{i}+R_{i+1}\le S_{i-1}+S_{i} $.
		
   		\item[\rm (ii)] For any $ 1\le i\le \min\{m-1,n\} $, we have $ d[a_{1,i}b_{1,i}]\ge A_{i} $.
   		
		\item[\rm (iii)] For any $ 1<i\le \min\{m-1,n+1\} $, if
		\begin{align}\label{eq:assumption(3)(ii)'}
			R_{i+1}>S_{i-1} \quad
			\text{and}\quad d[-a_{1,i}b_{1,i-2}]+d[-a_{1,i+1}b_{1,i-1}]>2e+S_{i-1}-R_{i+1},
		\end{align}
		then $ [b_{1},\ldots, b_{i-1}]\rep [a_{1},\ldots,a_{i}] $.
		
		\item[\rm (iv)] For any $ 1<i\le \min\{m-2,n+1\} $ such that $ S_{i}\ge R_{i+2}>S_{i-1}+2e\ge R_{i+1}+2e$, we have $ [b_{1},\ldots,b_{i-1}]\rep [a_{1},\ldots,a_{i+1}] $. (If $ i=n+1 $, the condition $ S_{i}\ge R_{i+2} $ is ignored.)
	\end{enumerate} 
\end{thm}
Recall from \cite[Definition 3.1]{HeHu2} that for $ c\in F^{\times}\backslash (F^{\times 2}\cup \Delta F^{\times 2}) $, write $\delta=c\pi^{-\ord(c)}\in \mathcal{O}_{F}^{\times}$ and let $\delta=s^{2}(1+r\pi^{d(c)})$,
with $ r,s\in \mathcal{O}_{F}^{\times} $, when $\ord (c) $ is even. Then we put
\begin{align}\label{defn:csharp}
	c^{\#}:=
	\begin{cases}
		\Delta    &\text{if  $ \ord (c) $ is odd,}\\
		1+4\rho r^{-1}\pi^{-d(c)}  &\text{if $ \ord (c) $ is even.}
	\end{cases}
\end{align}
\begin{defn}\label{defn:classicallattices}
 Let $ n $ be an integer and $ n\ge 2 $. Write $ \omega=1+\pi $ (and thus $ \omega^{\#}=1+4\rho\pi^{-1} $ by \eqref{defn:csharp}). For $c\in F^{\times}/F^{\times 2}$, if $n$ is even, we define the rank $n$ lattices
 \begin{align*}
 	 H_{e}^{n}(c)&:=\mathbf{H}_{e}^{(n-2)/2}\perp \prec \pi^{e},-c\pi^{-e} \succ,\\
 	C_{1}^{n}(c)&:=\mathbf{H}_{0}^{(n-2)/2}\perp \prec 1,-c\succ,  \\
 	C_{2}^{n}(c)&:=\mathbf{H}_{0}^{(n-2)/2}\perp \prec c^{\#},-c^{\#}c\succ.  
 \end{align*}
If $n$ is odd, then we define the rank $n$ lattices
 \begin{align*}
 	C_{1}^{n}(c)&:=\mathbf{H}_{0}^{(n-1)/2}\perp \prec c\succ,  \\
 		C_{2}^{n}(c)&:=\mathbf{H}_{0}^{(n-3)/2}\perp \begin{cases}
 		\prec 1,-\Delta ,c\Delta\succ   &\text{if $\ord (c) $ is odd},  \\
 		\prec c\omega^{\#},-c\omega^{\#}\omega,c\omega \succ &\text{if $\ord (c) $ is even}. \\ 
 	\end{cases}   
 \end{align*}
 For each fixed $e\ge 1$, define the set $\mathcal{C}_{e}^{n}$ of the rank $n$ lattices listed in the following table.
 	\begin{center}
 	\renewcommand\arraystretch{1.8}
 	\begin{tabular}{c|c|c|c|c}
 	 \toprule[1.2pt]
		 $ n $& $e$ &  $c$   &   $H_{e}^{n}(c)$    &    \\
 		 \cline{1-4}	 	
 		\multirow{2}*{\text{Even}}	 	
 		& $1$ &$\delta\in \{1,\Delta\}$ & $\mathbf{H}_{1}^{\frac{n-2}{2}}\perp \prec \pi,-\delta\pi^{-1}\succ$ & \\
 		\cline{2-4}
 		& $\ge 2 $ &$\delta=1$ & $\mathbf{H}_{e}^{\frac{n}{2}}$  & \\
 	 \midrule[1pt]
 	     $n$ & \multicolumn{2}{c|}{$c$}   & $ C_{1}^{n}(c) $  & $ C_{2}^{n}(c) $  \\
 	     \hline
 	   	\multirow{2}*{\text{Even}}	 	
 		& \multicolumn{2}{c|}{$\delta\in \mathcal{U}_{1} $}  & $ \mathbf{H}_{0}^{\frac{n-2}{2}}\perp \prec 1,-\delta \succ$ & $ \mathbf{H}_{0}^{\frac{n-2}{2}}\perp \prec\delta^{\#},-\delta^{\#}\delta \succ$  \\ 
 		\cline{2-5}
 		& \multicolumn{2}{c|}{$ \delta\pi$, with $\delta\in \mathcal{U} $}   & $ \mathbf{H}_{0}^{\frac{n-2}{2}}\perp \prec  1,-\delta\pi\succ$ & $ \mathbf{H}_{0}^{\frac{n-2}{2}}\perp \prec\Delta,-\Delta\delta\pi\succ$  \\
 		\hline
 		\multirow{2}*{\text{Odd}} & \multicolumn{2}{c|}{$\delta\in\mathcal{U} $}  & $ \mathbf{H}_{0}^{\frac{n-1}{2}}\perp \prec \delta\succ $  &   $ \mathbf{H}_{0}^{\frac{n-3}{2}}\perp 	\prec \delta\omega^{\#},-\delta\omega^{\#}\omega,  \delta\omega \succ  $ \\
 		\cline{2-5}
 		& \multicolumn{2}{c|}{$ \delta\pi$, with $\delta\in\mathcal{U} $}   & $ \mathbf{H}_{0}^{\frac{n-1}{2}}\perp \prec\delta\pi\succ $  &   $ \mathbf{H}_{0}^{\frac{n-3}{2}}\perp \prec 1,-\Delta ,\Delta\delta\pi\succ $  \\
 		\bottomrule[1.2pt]
 	\end{tabular}
 \end{center}
\end{defn}
 The following lemma allows us to obtain the invariants of a lattice by putting together those of its components.
\begin{lem}\label{lem:BONGformaximallattice}
 	Let $ s,t $ be integers with $s\ge 0$ and $0\le t\le e$. Let $ L $ be an $ \mathcal{O}_{F} $-lattice of rank $ \ell $. If $ L\cong \prec c_{1},\ldots,c_{\ell} \succ $ relative to a good BONG and $\ord\,(c_{1})\ge t$, then $ \mathbf{H}_{t}^{s}\perp L\cong \prec \pi^{t},-\pi^{-t},\ldots, \pi^{t},-\pi^{-t}, c_{1},\ldots,c_{\ell}\succ $ relative to a good BONG.
\end{lem}
\begin{proof}
	Similar to \cite[Lemma 3.10]{HeHu2}.
\end{proof}
\begin{prop}\label{prop:translation}
	\begin{enumerate}
		\item[\rm (i)] The lattices in $\mathcal{C}_{e}^{n}$ are precisely the lattices listed in Theorem \ref{thm:classicalnuniversaldyadic15theorem}.
		
		\item[\rm (ii)]  We have
		\begin{align*}
			|\mathcal{C}_{e}^{n}|=
			\begin{cases}
				8(N\mathfrak{p})^{e}-4 (N\mathfrak{p})^{e-1}+1   &\text{if $n$ is even and $e>1$},\\
				8(N\mathfrak{p})^{e}-4 (N\mathfrak{p})^{e-1}+2    &\text{if $n$ is even and $e=1$},  \\
				8(N\mathfrak{p})^{e} &\text{if $n$ is odd}.
			\end{cases}
		\end{align*}
		
		\item[\rm (iii)] All lattices in $\mathcal{C}_{e}^{n}$ are classic integral.
	\end{enumerate} 
\end{prop}
\begin{proof}
	(i) By \cite[Corollary 2.3(ii)]{HeHu2}, we have
	\begin{align*}
	 \prec \pi^{e},-\pi^{-e}\succ\cong A(0,0)\quad\text{and}\quad  \prec \pi^{e},-\Delta\pi^{-e}\succ \cong A(2,2\rho).
	\end{align*}
	By \cite[Lemma 3.9(ii)]{HeHu2}, we have
 \begin{align*}
		\prec 1,\,-\delta\pi\succ\cong 	\langle 1,\,-\delta \pi\rangle \quad\text{and}\quad \prec \Delta,\,-\Delta\delta\pi\succ\cong \langle \Delta,\,-\Delta\delta \pi\rangle,\;\text{with\; $\delta\in \mathcal{O}_{F}^{\times}$}.
	\end{align*}
	 By \cite[Remark 3.8]{HeHu2} with $\delta\in\mathcal{O}_{F}^{\times}$ and $d(\delta)=1$, we have
		\begin{align*}
		\prec 1,-\delta\succ \cong  A(1,-(\delta-1)) \quad\text{and}\quad
		\prec \delta^{\#},-\delta^{\#}\delta\succ \cong \delta^{\#}  A(1,-(\delta-1)),
	\end{align*}
			where $ \delta^{\#}=1+4\rho(\delta-1)^{-1} $.
	In particular, we have $ \prec \omega^{\#},-\omega^{\#}\omega\succ \cong \omega^{\#}  A(1,-(\omega-1))$.
	So
	\begin{align*}
			\prec \delta\omega^{\#},-\delta\omega^{\#}\omega,\delta\omega\succ \cong \prec \delta\omega^{\#},-\delta\omega^{\#}\omega\succ\perp \prec\delta\omega\succ  \cong \delta\omega^{\#} A(1,-(\omega-1))\perp \langle \delta\omega \rangle
	\end{align*}
	by \cite[Lemma 4.3(iii)]{beli_integral_2003}.
	
	(ii) For even $n\ge 2$, note that the number of all units $ \varepsilon $ (in $ \mathcal{O}_{F}^{\times 2} $) with $ d(\varepsilon)>1 $ is $ 2(N\mathfrak{p})^{e-1} $. By \cite[63:5, 63:9]{omeara_quadratic_1963}, we count the number of the lattices in $ \mathcal{C}_{e}^{n} $ as follows, 
	\begin{align*}
		|\mathcal{C}_{e}^{n}|&=1+2|\mathcal{U}|+2|\mathcal{U}_{1}|=1+4(N \mathfrak{p})^{e}+(4(N\mathfrak{p})^{e}-4(N\mathfrak{p})^{e-1})= 8(N\mathfrak{p})^{e}-4(N\mathfrak{p})^{e-1}+1,
	\end{align*}
	when $ e>1 $. Similarly, we have $ 8(N\mathfrak{p})^{e}-4(N\mathfrak{p})^{e-1}+2$ when $ e=1 $.
	
	For odd $n\ge 3$, the spaces spanned by the lattices in $ \mathcal{C}_{e}^{n} $ exhaust all the possible quadratic spaces of dimension $ n $  (cf. \cite[Proposition 3.5(i)]{HeHu2}) and hence $ |\mathcal{C}_{e}^{n}|=4[\mathcal{O}_{F}^{\times}:\mathcal{O}_{F}^{\times 2}]=8(N\mathfrak{p})^{e} $.
	
	(iii) This follows by \eqref{eq:integralcondition}, \eqref{eq:classicallyintegral} and Lemma \ref{lem:S-invariant}(i)(ii) below.
\end{proof}
\begin{lem}\label{lem:S-invariant}
	Let $ N $ be an $ \mathcal{O}_{F} $-lattice of rank $ n\ge 2 $, $ S_{i}=R_{i}(N) $ and $\beta_{i}=\alpha_{i}(N)$.
	\begin{enumerate}
		\item[\rm (i)] Suppose that $n$ is even.
		
		\item[] If $ N=H_{e}^{n}(1) $ or $ H_{e}^{n}(\Delta) $, then $ S_{i}=e $ for $ i\in [1,n]^{O} $ and $ S_{i}=-e $ for $ i\in [1,n]^{E} $.
		
		\item[] If $ N=C_{1}^{n}(c)$ or $ C_{2}^{n}(c) $ with $c\in F^{\times}/F^{\times 2}$ and $d(c)\in \{0,1\}$, then $ S_{i}=0 $  for  $1\le i\le n-1$ and $S_{n}=1-d(c)$.
		
		\item[\rm (ii)] Suppose that $ n $ is odd.
		
		\item[] If $ N=C_{1}^{n}(c)$ or $ C_{2}^{n}(c) $ with $c\in F^{\times}/F^{\times 2}$ and $\ord\,(c)\in \{0,1\}$, then $ S_{i}=0 $  for  $1\le i\le n-1$ and $S_{n}=\ord\,(c)$.
		
		\item[\rm (iii)]  If $N=C_{1}^{n}(c)$ or $C_{2}^{n}(c)$ with $c\in F^{\times}/F^{\times 2}$ and $d(c)\in \{0,1\}$, then $ S_{i}=0 $  for  $1\le i\le n-1$, $S_{n}=1-d(c)$ and $\beta_{i}=1$ for $1\le i\le n-1$.
	\end{enumerate}
\end{lem}
\begin{proof}
 (i) and (ii) follows by Lemma \ref{lem:BONGformaximallattice}, Proposition \ref{prop:translation}(i) and \cite[Corollary 4.4(i)]{beli_integral_2003}. For (iii), we have $S_{i}=0$ for $1\le i\le n-1$ and 
 \begin{align*}
 	S_{n}=
 	\begin{cases}
 		0   &\text{if $d(c)=1$},  \\
 		1   &\text{if $d(c)=0$}
 	\end{cases}
 \end{align*}
 by (i) and (ii), noticing that $\ord\,(c)$ is odd iff $d(c)=0$. Thus $S_{n}=1-d(c)$. If $d(c)=1$, then $S_{n}-S_{n-1}=0$ and so $ \beta_{n-1}\le S_{n}-S_{n-1}+d(-b_{n-1}b_{n})=d(c)=1$ by \eqref{eq:alpha-defn}. If $d(c)=0$, then $S_{n}-S_{n-1}=1$ and so $\beta_{n-1}=1$ by Proposition \ref{prop:alphaproperty}(iii). Hence, in both cases, $\beta_{i}=1$ for $1\le i\le n-1$ by Proposition \ref{prop:Ralphaproperty3}(vi).
\end{proof}
\begin{prop}\label{prop:isotropic-ternary}
	The ternary quadratic space $[a_{1},a_{2},a_{3}]$ over $F$ is isotropic or anisotropic, according as the Hilbert symbol $(-a_{1}a_{2},-a_{1}a_{3})_{\mathfrak{p}}=1$ or $-1$.
\end{prop}
\begin{proof}
Write $V:=[a_{1},a_{2},a_{3}]$ and $S_{\mathfrak{p}}V$ for the Hasse symbol of $V$. From \cite[p.\,152]{omeara_quadratic_1963}, we have $(-1,-1)_{\mathfrak{p}}S_{\mathfrak{p}}V=(-a_{1}a_{2},-a_{1}a_{3})_{\mathfrak{p}}$. Hence $S_{\mathfrak{p}}V=(-1,-1)_{\mathfrak{p}}$ or $-(-1,-1)_{\mathfrak{p}}$, according as $(-a_{1}a_{2},-a_{1}a_{3})_{\mathfrak{p}}=1$ or $-1$. So the proposition follows by \cite[58:6]{omeara_quadratic_1963}.
\end{proof}
For $ s \in \mathbb{N} $,  we write $ \mathbb{H}^{s} $ for the orthogonal sum of $s $ copies of the hyperbolic plane $ \mathbb{H} $.
\begin{lem}\label{lem:spacerep-criterion}
	Let $n$ be an even integer and $n\ge 2$. Let $V$ be a quadratic space over $F$ and $\dim V=n+1$. 
	\begin{enumerate}
		\item[\rm (i)]  If $\det V=\varepsilon\in \mathcal{O}_{F}^{\times }$ (in $F^{\times 2}$), then $V$ cannot represent both of $\mathbb{H}^{n/2}$ and $\mathbb{H}^{(n-2)/2}\perp [\pi,-\Delta\pi]$.
		
		 \item[] Suppose $V=\mathbb{H}^{(n-2)/2}\perp U$ for some ternary space $U$ with $\det U=\varepsilon\in \mathcal{O}_{F}^{\times}$.

		\item[\rm (ii)] If $U$ is isotropic, then $V$ represents $\mathbb{H}^{n/2}$, but does not represent $\mathbb{H}^{(n-2)/2}\perp [\pi,-\Delta\pi]$.
		
		\item[\rm (iii)] If $U$ is anisotropic, then $V$ represents $\mathbb{H}^{(n-2)/2}\perp [\pi,-\Delta\pi]$, but does not represent $\mathbb{H}^{n/2}$.		
	\end{enumerate} 
\end{lem}
\begin{proof}
	For (i), assume to the contrary that $V $ represents both $\mathbb{H}^{n/2}$ and $\mathbb{H}^{(n-2)/2}\perp [\pi,-\Delta\pi]$. Then, by \cite[63:21]{omeara_quadratic_1963}, we have
	\begin{align*}
		\mathbb{H}^{(n-2)/2}\perp [1,-1]\perp [(-1)^{n/2}\varepsilon]\cong V\cong \mathbb{H}^{(n-2)/2}\perp [\pi,-\Delta\pi]	 \perp [(-1)^{n/2}\Delta\varepsilon],
	\end{align*}
	which implies $[1,-1,(-1)^{n/2}\varepsilon]\cong   [\pi,-\Delta\pi,(-1)^{n/2}\Delta\varepsilon]$
	by Witt's cancellation. Since $\varepsilon$ is a unit, by Proposition \ref{prop:isotropic-ternary}, these two spaces have opposite isotropy, a contradiction. This shows (i).
	
	For (ii) and (iii), we may assume $n=2$ by Witt's cancellation. If $U$ is isotropic, then $U\cong \mathbb{H}\perp [\varepsilon]$ represents $\mathbb{H}$, but does not represent $[\pi,-\Delta\pi]$ by (i). If $U$ is anisotropic, then clearly $U$ cannot represent $\mathbb{H}$. Since $\det U=\varepsilon$, $U\cong [\pi, -\Delta\pi,-\Delta\varepsilon]$ by \cite[Proposition 3.5]{HeHu2}, representing $[\pi,-\Delta\pi]$.
\end{proof}


\section{Sufficient conditions for $N\rep M$}\label{sec:sufficient conditions}
In the remainder of the paper, all lattices under consideration will be assumed to be classic integral. Thus ``$n$-universal" means ``classic $n$-universal".

Let $ M\cong \prec a_{1},\ldots,a_{m} \succ$ be an $ \mathcal{O}_{F} $-lattice of rank $ m\ge n+1\ge 3 $ relative to some good BONG. Let $ R_{i}=R_{i}(M) $ for $ 1\le i\le m $ and $ \alpha_{i}=\alpha_{i}(M) $ for $ 1\le i\le m-1 $. For convenience, whenever an $ \mathcal{O}_{F} $-lattice $ N $ of rank $ n $ is discussed, we always assume $ N\cong \prec b_{1},\ldots, b_{n}\succ $ relative to some good BONG and denote by $ S_{i}=R_{i}(N) $ and $ \beta_{i}=\alpha_{i}(N) $.
\begin{lem}\label{lem:B1-B3-B4-j}
	Let $j$ be a positive integer.
	\begin{enumerate}
		\item[\rm (i)] If $j$ is odd and $R_{j}=0$, then Theorem \ref{thm:beligeneral}(i) holds at the index $j$.
		
		\item[\rm (ii)] If $1<j<m$ and $R_{j}+R_{j+1}=0$, then Theorem \ref{thm:beligeneral}(i) holds at the index $j$.
		
		\item[\rm (iii)] If $j$ is even and $R_{j+1}=0$, then $j$ is not essential (in the sense of \cite[Definition 4.7]{beli_representations_2006}). Thus Theorem \ref{thm:beligeneral}(iii) holds at the index $j$.
		
		\item[\rm (iv)] If $1<j<m$ and $R_{j+1}+R_{j+2}=0$, then $j$ is not essential. Thus Theorem \ref{thm:beligeneral}(iii) holds at the index $j$.

		\item[\rm (v)] If $R_{j+2}-R_{j+1}\le 2e$, then Theorem \ref{thm:beligeneral}(iv) holds at the index $j$.
	\end{enumerate} 
\end{lem}
\begin{proof}
	  Since $j$ is odd, by Proposition \ref{prop:Ralphaproperty3}(i), we have $R_{j}= 0\le S_{j}$ and thus (i) is proved. By Proposition \ref{prop:Ralphaproperty3}(i) again, we have $R_{j}+R_{j+1}=0\le S_{j}+S_{j-1}$. This shows (ii). 
	  
	  For (iii) and (iv), we similarly have $R_{j+1}=0\le S_{j-1}$ and $R_{j+1}+R_{j+2}=0\le S_{j-2}+S_{j-1}$, respectively. Hence the index $ j $ is not essential in both cases, so Theorem \ref{thm:beligeneral}(iii) holds trivially at $ j $ by \cite[Lemma 4.9]{beli_representations_2006}. For (v), it is trivial.
\end{proof}
\begin{lem}\label{lem:B2-j}
	 Let $j$ be a positive integer. Suppose $R_{i}=0$ for $1\le i\le j+1$.
	
	 If $j\in [1,m-1]^{E}$, then Theorem \ref{thm:beligeneral}(ii) holds at the index $i$ with $2\le i\le \min\{j-1,n\}$.
\end{lem}
\begin{proof}
    From Lemma \ref{lem:B1-B3-B4-j}(iii) and (iv), we see that the indices $ 2,\ldots, j $ are not essential. Hence Theorem \ref{thm:beligeneral}(ii) holds for $ i=2,\ldots, j-1 $ by \cite[Lemma 4.8]{beli_representations_2006}.
\end{proof}
\begin{lem}\label{lem:A1}
  If $R_{1}=R_{2}=0$ and $\alpha_{2}=1$, then Theorem \ref{thm:beligeneral}(ii) holds at the index $i=1$.
\end{lem}
\begin{proof}
	 If $ S_{1}\ge 1 $, since $R_{2}=0$ and $ \alpha_{2}=1 $, we have
	\begin{align*}
		A_{1}\le R_{2}-S_{1}+d[-a_{1,2}]\le R_{2}-S_{1}+\alpha_{2}=0-S_{1}+1\le 0\le d[a_{1}b_{1}].
	\end{align*}
	 If $ S_{1}=0 $, then $S_{2}-S_{1}=S_{2}\ge -e$ by \eqref{eq:classicallyintegralR2}. Hence Proposition \ref{prop:alphaproperty}(ii) implies $ \beta_{1}\ge 1 $. Since $R_{1}=0$, $\ord(a_{1}b_{1})$ is even and thus $ d(a_{1}b_{1})\ge 1 $. Combining these with $\alpha_{1}=1$, we have $d[a_{1}b_{1}]=\min\{d(a_{1}b_{1}),\alpha_{1},\beta_{1}\}=1$. Hence
		\begin{align*}
		A_{1}\le R_{2}-S_{1}+d[-a_{1,2}]\le R_{2}-S_{1}+\alpha_{2}=0-0+1=1=d[a_{1}b_{1}].
	\end{align*}
\end{proof}
\begin{lem}\label{lem:An-even}
Let $  j\in[ 1,\min\{m-1,n\}]^{E} $. Assume that either $e=1$, or $e>1$ and $ d[(-1)^{(j+2)/2}a_{1,j+2}]\le 1-R_{j+2} $. If $R_{j}=R_{j+1}=0$ and $\alpha_{j}=1$, then Theorem \ref{thm:beligeneral}(ii) holds at the index $ j $.
\end{lem}
\begin{proof}
 Since $R_{j}=R_{j+1}=0$, Proposition \ref{prop:Ralphaproperty3}(v) implies that $R_{i}=0$ for $1\le i\le j+1$.
	 Now, we have
	\begin{align*}
		A_{j}=\min\{\dfrac{-S_{j}}{2}+e,-S_{j}+d[-a_{1,j+1}b_{1,j-1}],R_{j+2}-S_{j-1}-S_{j}+d[a_{1,j+2}b_{1,j-2}]\}.
	\end{align*}
We may assume $ S_{j}<\min\{d[-a_{1,j+1}b_{1,j-1}],2e\} $ (otherwise, $ A_{j}\le 0\le d[a_{1,j}b_{1,j}] $). Since $j-1$ is odd, $S_{j-1}\ge 0$ by Proposition \ref{prop:Ralphaproperty3}(i). Then, by the assumption and \eqref{eq:alpha-defn}, we have
\begin{align*}
	S_{j}<d[-a_{1,j+1}b_{1,j-1}]\le 	\beta_{j-1}\le S_{j-1}+\beta_{j-1}\le  S_{j}+d[-b_{j-1}b_{j}],
\end{align*}
which implies $ d[-b_{j-1}b_{j}]>0 $. It follows that $\beta_{j}\ge 1$ (if $\beta_{j}$ exists). Since $\alpha_{j}=1$, we have $ d[a_{1,j}b_{1,j}]=\min\{d(a_{1,j}b_{1,j}),1\} $ regardless of the existence of $\beta_{j}$. 

Firstly, if $j>2$, then, by the domination principle, we have
	\begin{align}\label{eq:bj-1bj}
		d[-b_{k-1}b_{k}]=\min_{i\in [1,j-2]^{E}}\{d[-b_{i-1}b_{i}]\}\le d[(-1)^{(j-2)/2}b_{1,j-2}]
	\end{align} 
	for	some $ k\in [1,j-2]^{E} $. Next, we consider two cases according as the parity of the sum $\sum_{i=1}^{j}S_{i}$.
	
	\textbf{Case I: $ \sum\limits_{i=1}^{j}S_{i} $ is odd}
	
	Since $\ord(a_{1,j}b_{1,j})$ is odd, $ d[a_{1,j}b_{1,j}]=0 $. Note that Proposition \ref{prop:Rproperty}(vi) implies $\alpha_{i}=\alpha_{j}=1$ for $1\le i\le j$. Then for $i\in [1,j]^{E}$, since $\ord(a_{i-1}a_{i})$ is even, $d(-a_{i-1}a_{i})>0$ and so $d[-a_{i-1}a_{i}]=\min\{d(-a_{i-1}a_{i}),\alpha_{i-2},\alpha_{i}\}>0$. Combining with $d[-b_{j-1}b_{j}]>0$, we have $ d[(-1)^{(j-2)/2}b_{1,j-2}]=d[a_{1,j}b_{1,j}]=0$ by the domination principle. Since $ \sum_{i=1}^{j}S_{i} $ is assumed to be odd, $ j>2 $ and thus \eqref{eq:bj-1bj} holds. Also, $S_{k-1}\ge 0$ by Proposition \ref{prop:Ralphaproperty3}(i). Hence we conclude that
	\begin{align*}
		-S_{j}+\beta_{j-1}\le -S_{k}+\beta_{k-1}\le S_{k-1}-S_{k}+\beta_{k-1}\underset{\eqref{eq:alpha-defn}}{\le} d[-b_{k-1}b_{k}]\underset{\eqref{eq:bj-1bj}}{\le} d[(-1)^{(j-2)/2}b_{1,j-2}]=0,
	\end{align*} 
	where the first inequality holds by Proposition \ref{prop:Rproperty}(v). So $ \beta_{j-1}\le S_{j}<d[-a_{j+1}b_{1,j-1}] $, which contradicts $ d[-a_{1,j+1}b_{1,j-1}]\le \beta_{j-1} $.
	
	\textbf{Case II: $ \sum\limits_{i=1}^{j}S_{i} $ is even}
	
	Since $\ord(a_{1,j}b_{1,j})$ is even, $ d(a_{1,j}b_{1,j})\ge 1 $ and thus $ d[a_{1,j}b_{1,j}]=\min\{d(a_{1,j}b_{1,j}),1\}=1 $.
	
	\textbf{SubCase I: $ e>1 $} 
	
 	If $d[(-1)^{(j+2)/2}a_{1,j+2}]\not=d[(-1)^{(j-2)/2}b_{1,j-2}]$, then 
	 \begin{align*}
 		d[a_{1,j+2}b_{1,j-2}]=\min\{d[(-1)^{(j+2)/2}a_{1,j+2}],d[(-1)^{(j-2)/2}b_{1,j-2}]\}\le d[(-1)^{(j+2)/2}a_{1,j+2}].
	 \end{align*}
	Hence	 
	\begin{align*}
		A_{j}&\le R_{j+2}+d[a_{1,j+2}b_{1,j-2}]  \quad\text{(as $ S_{j}+S_{j-1}\ge 0 $ by Proposition \ref{prop:Ralphaproperty3}(i))}\\
		&\le R_{j+2}+d[(-1)^{(j+2)/2}a_{1,j+2}]  \\
		&\le R_{j+2}+(1-R_{j+2})=1 \quad\text{(by the hypothesis)}.
	\end{align*}
	Note that if $j=2$, then $b_{1,j-2}=1$ and the argument also holds in this case.
	
	If $d[(-1)^{(j+2)/2}a_{1,j+2}]=d[(-1)^{(j-2)/2}b_{1,j-2}]$, we may assume $j>2$. Then
	\begin{align*}
		d[-b_{k-1}b_{k}]\underset{\eqref{eq:bj-1bj}}{\le}d[(-1)^{(j-2)/2}b_{1,j-2}]=d[(-1)^{(j+2)/2}a_{1,j+2}]\le 1-R_{j+2}.
	\end{align*}
	Since both $k-1$ and $j+2$ is odd, we have $S_{k-1}\ge 0$ and $R_{j+2}\ge 0$ by Proposition \ref{prop:Ralphaproperty3}(i). Hence
	\begin{align*}
		A_{j}\le -S_{j}+d[-a_{1,j+1}b_{1,j-1}]&\le -S_{j}+\beta_{j-1}\\
		&\le  -S_{k}+\beta_{k-1} \quad\text{(by Proposition \ref{prop:Rproperty}(v))}\\
		&\underset{\eqref{eq:alpha-defn}}{\le} -S_{k-1}+d[-b_{k-1}b_{k}]\le 1-R_{j+2}\le 1.
	\end{align*}
	
	\textbf{SubCase II: $ e=1 $}
	
	In this case, $   S_{j}\ge -e=-1$ by Proposition \ref{prop:Ralphaproperty3}(i). If $ S_{j}\ge 0 $, then
	\begin{align*}
		A_{j}\le \dfrac{-S_{j}}{2}+e\le 0+1=1.
	\end{align*}
	If $ S_{j}=-1 $, then $S_{j-1}=1$ by Proposition \ref{prop:Ralphaproperty3}(iii). Thus $ S_{j}-S_{j-1}=-2=-2e $ and so $ d[-a_{1,j+1}b_{1,j-1}]=\beta_{j-1}=0 $ by Proposition \ref{prop:alphaproperty}(ii). Hence
	\begin{align*}
		A_{j}\le -S_{j}+d[-a_{1,j+1}b_{1,j-1}]=-S_{j} =1.
	\end{align*}	
	With above discussion, we conclude that $ A_{j}\le d[a_{1,j}b_{1,j}] $.
\end{proof}
\begin{lem}\label{lem:d[-ai+1ai+2]=1-Ri+2}
  Let $  i\in [1,\min\{m-2,n\}]^{E} $. If $R_{i}=R_{i+1}=0$ and $ d[(-1)^{(i+2)/2}a_{1,i+2}]=1-R_{i+2} $, then one of the following statements holds:
	\begin{enumerate}
		\item[\rm (i)] $ d[-a_{1,i+2}b_{1,i}]=1-R_{i+2} $.
		
		\item[\rm (ii)] There exists some $ j\in [1,i]^{E} $ such that $ R_{i+2}\le S_{j} $ and $ \beta_{k}\le S_{k+1}-S_{j-1}+1-R_{i+2}\le S_{k+1}+1-R_{i+2} $ for each $ j-1\le k\le n-1 $.
	\end{enumerate}
\end{lem}
\begin{proof}
	Similar to \cite[Lemma 2.10(iii)]{HeHu2}.
\end{proof}
\begin{lem}\label{lem:An-odd}
	  Suppose that $n\ge 3$ is odd. If $ R_{n-1}=R_{n}=0 $, $ \alpha_{n}=1 $ and $d[(-1)^{(n+1)/2}a_{1,n+1}]=1-R_{n+1}$, then $R_{n+1}-S_{n}+d[-a_{1,n+1}b_{1,n-1}]\le   d[a_{1,n}b_{1,n}]$. Thus Theorem \ref{thm:beligeneral}(ii) holds at the index $ n $. 
\end{lem}
\begin{proof}
	 Since $R_{n-1}=R_{n}=0$, Proposition \ref{prop:Ralphaproperty3}(v) implies that $R_{i}=0$ for $1\le i\le n$. Then the argument is similar to \cite[Lemma 2.11]{HeHu2}.
\end{proof}
\begin{lem}\label{lem:An-odd-2}
	Suppose that $n\ge 3$ is odd. If $ R_{n-1}=R_{n}=0 $, $ \alpha_{n}=1 $ and $d[(-1)^{(n+1)/2}a_{1,n+1}]=1-R_{n+1}$, then $d[-a_{1,n+1}b_{1,n-1}] \le  S_{n}-R_{n+1}+d[a_{1,n}b_{1,n}]\le S_{n}-R_{n+1}+1$.
	
	If moreover $d[-a_{1,n+2}b_{1,n}]=0$ and $d[-a_{1,n+1}b_{1,n-1}]+d[-a_{1,n+2}b_{1,n}]>2e+S_{n}-R_{n+2}$, then $ R_{n+2}-R_{n+1}>2e-1  $.
\end{lem}
\begin{proof}
	Since $\alpha_{n}=1$, Lemma \ref{lem:An-odd} implies that $ R_{n+1}-S_{n}+d[-a_{1,n+1}b_{1,n-1}]\le d[a_{1,n}b_{1,n}]\le \alpha_{n}=1 $. Rewriting the inequality yields the first statement.
		  
	From the assumption, we have
		  \[
		  d[-a_{1,n+1}b_{1,n-1}]=d[-a_{1,n+1}b_{1,n-1}]+d[-a_{1,n+2}b_{1,n}]>2e+S_{n}-R_{n+2}.
		  \]
	Combining with the first statment, we see that $ R_{n+2}-R_{n+1}>2e-1 $.
\end{proof}
\begin{lem}\label{lem:B3-n+1}
	Suppose that $n\ge 2$ is even. If $R_{n}=R_{n+1}=0$, $\alpha_{n+1}=1$ and $d[(-1)^{(n+2)/2}a_{1,n+2}]=1-R_{n+2}$, then Theorem \ref{thm:beligeneral}(iii) holds at the index $n+1$.
\end{lem}	
\begin{proof}
	Assume that $ R_{n+2}>S_{n} $ and $ d[-a_{1,n+1}b_{1,n-1}]+d[-a_{1,n+2}b_{1,n}]>2e+S_{n}-R_{n+2} $. If $ d[-a_{1,n+2}b_{1,n}]\not=1-R_{n+2} $, then $ R_{n+2}\le S_{j}\le S_{n} $ for some $ j\in [1,n]^{E} $ by Lemma \ref{lem:d[-ai+1ai+2]=1-Ri+2} with $i=n$ and Proposition \ref{prop:Ralphaproperty3}(i), a contradiction. Hence $ d[-a_{1,n+2}b_{1,n}]=1-R_{n+2} $. Since  $\alpha_{n+1}=1$, we have $d[-a_{1,n+1}b_{1,n-1}]\le  \alpha_{n+1}=1 $. Hence
	\begin{align*}
		2-R_{n+2}=1+(1-R_{n+2})\ge d[-a_{1,n+1}b_{1,n-1}]+d[-a_{1,n+2}b_{1,n}]>2e+S_{n}-R_{n+2},
	\end{align*}
	which implies $ S_{n}<2-2e\le 0 $. Hence $ S_{n-1}>0$ by the second part of Proposition \ref{prop:Ralphaproperty3}(i) and so
	\begin{align*}
		-2e\le S_{n}-S_{n-1}<S_{n}<2-2e
	\end{align*}
	 by \eqref{eq:BONGs}. This implies $S_{n}=1-2e$ and $S_{n}-S_{n-1}=-2e$. Hence $ d[-a_{1,n+1}b_{1,n-1}]=\beta_{n-1}=0 $ by Proposition \ref{prop:alphaproperty}(ii). Therefore, we deduce that
	\begin{align*}
		1-R_{n+2}=0+(1-R_{n+2})=d[-a_{1,n+1}b_{1,n-1}]+d[-a_{1,n+2}b_{1,n}]>2e+S_{n}-R_{n+2},
	\end{align*}
	which implies $ S_{n}<1-2e $, a contradiction.
\end{proof}
\begin{lem}\label{lem:B3-n-n+1-odd}
	Suppose that  $n\ge 3$ is odd. If $R_{n-1}=R_{n}=0$, $\alpha_{n}=1$ and $d[(-1)^{(n+1)/2}a_{1,n+1}]=1-R_{n+1}$.
	\begin{enumerate}
		\item[\rm (i)] Theorem \ref{thm:beligeneral}(iii) holds at index $i=n$.
		
		\item[\rm (ii)] If moreover $R_{n+1}=0$ and $R_{n+2}\in \{0,1\}$, then Theorem \ref{thm:beligeneral}(iii) holds at index $i=n+1$.
	\end{enumerate}
\end{lem}	
\begin{proof}
	(i) Let $ N^{\prime}=\prec b_{1},\ldots,b_{n-1}\succ$ and $ \beta'_{i}=\alpha_{i}(N^{\prime})$ for $ 1\le i\le n-2 $. Comparing with \eqref{T} for $\beta_i=\alpha_i(N)$, we have 
	\[
	\beta_{i}=\min\{\beta_{i}^{\prime},S_{n}-S_{i}+d(-b_{n-1}b_{n})\}\le \beta'_i
	\]
	for $ 1\le i\le n-2 $. For $ 0\le i\le m $, $ 0\le j\le n-1 $ and $ c\in F^{\times} $, we denote by $ d^{\prime}[ca_{1,i}b_{1,j}] $ the invariant $ d[ca_{1,i}b_{1,j}] $ corresponding to $ M $ and $ N'$. Then
	\begin{align*}
		d[ca_{1,i}b_{1,j}] =\min\{d(ca_{1,i}b_{1,j}),\alpha_{i},\beta_{j}\} \quad\text{and}\quad  d^{\prime}[ca_{1,i}b_{1,j}] =\min\{d(ca_{1,i}b_{1,j}),\alpha_{i},\beta_{j}^{\prime}\}.
	\end{align*}
	Since $ \beta_{j}\le \beta_{j}^{\prime} $ for $ 1\le j\le n-2 $, we have $ d[ca_{1,i}b_{1,j}]\le d^{\prime}[ca_{1,i}b_{1,j}] $.
	
	Suppose $ R_{i+1}>S_{i-1} $ and $ d[-a_{1,i}b_{1,i-2}]+d[-a_{1,i+1}b_{1,i-1}]>2e+S_{i-1}-R_{i+1} $ for some $ 2\le i \le n $. Then
	\begin{align*}
		d^{\prime}[-a_{1,i}b_{1,i-2}]+d^{\prime}[-a_{1,i+1}b_{1,i-1}]\ge d[-a_{1,i}b_{1,i-2}]+d[-a_{1,i+1}b_{1,i-1}]>2e+S_{i-1}-R_{i+1}.
	\end{align*}
	By Lemma\;\ref{lem:B3-n+1}, Theorem \ref{thm:beligeneral}(iii) holds for $M$ and $ N'$ and so $ [b_{1},\ldots,b_{i-1}]\rep [a_{1},\ldots,a_{i}]$.
 
	(ii) Since $n$ is odd, we have $S_{n}\ge 0$ by Proposition \ref{prop:Ralphaproperty3}(i). 
	
	If $R_{n+2}=0$, then $R_{n+2}=0\le S_{n} $ and thus Theorem \ref{thm:beligeneral}(iii) holds trivially. 
	
	If $R_{n+2}=1$, assume that $R_{n+2}>S_{n}$ and $d[-a_{1,n+1}b_{1,n-1}]+d[-a_{1,n+2}b_{1,n}]>2e+S_{n}-R_{n+2}$. Then $S_{n}=0$ from the first assumption. Also, $R_{n}=0$ from the hypothesis. Hence Proposition \ref{prop:Ralphaproperty3}(ii) implies that both $\ord(b_{1,n})$ and $\ord(a_{1,n})$ are even. Since $R_{n+1}=0$ and $R_{n+2}=1$, $ \ord(a_{1,n+2}b_{1,n})$ is odd and so $ d[-a_{1,n+2}b_{1,n}]=0 $. Combining this with the second assumption, by Lemma \ref{lem:An-odd-2}, we see that $ 1=R_{n+2}-R_{n+1}>2e-1 $, i.e. $e<1$, which is impossible. 
\end{proof}
\begin{cor}\label{cor:B1}
	Suppose that $R_{i}=0$ for $1\le i\le n$.
	\begin{enumerate}
		\item[\rm (i)] If $n$ is even and either $R_{n+1}=0$ or $S_{n}\ge 0$, then Theorem \ref{thm:beligeneral}(i) holds for $1\le i\le n$.
		
		\item[\rm (ii)] If $n$ is odd, then Theorem \ref{thm:beligeneral}(i) holds for $1\le i\le n$. 
	\end{enumerate} 
\end{cor}
\begin{proof}
Suppose that $n$ is even and $R_{n+1}=0$. By Lemma \ref{lem:B1-B3-B4-j}(i), Theorem \ref{thm:beligeneral}(i) holds for $i\in [1,n]^{O}$. By Lemma \ref{lem:B1-B3-B4-j}(ii), Theorem \ref{thm:beligeneral}(i) holds for $i\in [1,n]^{E}$. Similarly for odd $n$.
	
Suppose that $n$ is even and $S_{n}\ge 0$. Then $n-1$ is odd. By (ii), Theorem \ref{thm:beligeneral}(i) holds for $1\le i\le n-1$. Since $R_{n}=0\le S_{n}$, Theorem \ref{thm:beligeneral}(i) also holds at $i=n$.
\end{proof}
\begin{cor}\label{cor:B2}
	Suppose  $R_{i}=0$ for $1\le i\le n$ and $\alpha_{i}=1$ for $1\le i\le n$.
	\begin{enumerate}
		\item[\rm (i)] Theorem \ref{thm:beligeneral}(ii) holds for $1\le i\le n-3$ or $1\le i\le n-2$, according as $n$ is even or odd.
		
		\item[\rm (ii)] If $n$ is even and $R_{n+1}=0$, then Theorem \ref{thm:beligeneral}(ii) holds for $1\le i\le n-1$; 
		
		\item[	] if moreover, $\alpha_{n+1}=1$, and either $e=1$, or $e>1$ and  $d[(-1)^{(n+2)/2}a_{1,n+2}]=1-R_{n+2}$, then Theorem \ref{thm:beligeneral}(ii) holds for $1\le i\le n$.
		
		\item[\rm (iii)] If $n$ is odd, $d[(-1)^{(n+1)/2}a_{1,n+1}]=1-R_{n+1}$, then Theorem \ref{thm:beligeneral}(ii) holds for $ 1\le i\le n$.
	\end{enumerate}  
\end{cor}
\begin{proof}
	(i)  Combine Lemma \ref{lem:A1} and Lemma \ref{lem:B2-j} with $j=n-2$ or $n-1$, according as $n$ is even or odd.
	
	(ii) Combine Lemma \ref{lem:A1}, Lemma \ref{lem:B2-j} with $j=n$ and Lemma \ref{lem:An-even} with $j=n$. 
	
	(iii) Combine (i), Lemma \ref{lem:An-even} with $j=n-1$ and Lemma \ref{lem:An-odd}.
\end{proof}
\begin{cor}\label{cor:B3}
	Suppose  $R_{i}=0$ for $1\le i\le n$.
	\begin{enumerate}
		\item[\rm (i)] Theorem \ref{thm:beligeneral}(iii) holds for $2\le i\le n-2$ or $2\le i\le n-1$, according as $n$ is even or odd. 
		
		\item[\rm (ii)]  If $n$ is even and $R_{n+1}=0$, then Theorem \ref{thm:beligeneral}(iii) holds for $2\le i\le n$; 
		
		\item[] if moreover, $\alpha_{n+1}=1$ and $d[(-1)^{(n+2)/2}a_{1,n+2}]=1-R_{n+2}$, then Theorem \ref{thm:beligeneral}(iii) holds for $2\le i\le n+1$.
		
		\item[\rm (iii)]   If $n$ is odd, $\alpha_{n}=1$ and $d[(-1)^{(n+1)/2}a_{1,n+1}]=1-R_{n+1}$, then Theorem \ref{thm:beligeneral}(iii) holds for $2\le i\le n$;
		
		\item[] if moreover, $R_{n+1}=0$ and $R_{n+2}\in \{0,1\}$, then Theorem \ref{thm:beligeneral}(iii) holds for $2\le i\le n+1$.
	\end{enumerate}   
\end{cor}
\begin{proof}
	(i) For even $n$, apply Lemma \ref{lem:B1-B3-B4-j}(iv) with $2\le j\le n-2$; for odd $n$, combine Lemma \ref{lem:B1-B3-B4-j}(iv) with $2\le j\le n-2$ and Lemma \ref{lem:B1-B3-B4-j}(iii) with $j=n-1$.
	
	(ii) Combine Lemma \ref{lem:B1-B3-B4-j}(iv) with $2\le j\le n-1$, Lemma \ref{lem:B1-B3-B4-j}(iii) with $j=n$ and Lemma \ref{lem:B3-n+1}.
	 
	(iii) Combine (i) and Lemma \ref{lem:B3-n-n+1-odd}(i) and (ii).
\end{proof}
\begin{cor}\label{cor:B4}
	Suppose  $R_{i}=0$ for $1\le i\le n$ and $R_{n+1}\in \{0,1\}$.
	\begin{enumerate}
		\item[\rm (i)] If $m=n+1$,  then Theorem \ref{thm:beligeneral}(iv) holds for $2\le i\le n-1$. 
		
		\item[\rm (ii)] If $m=n+2$ and $R_{n+2}-R_{n+1}\le 2e$, then Theorem \ref{thm:beligeneral}(iv) holds for $2\le i\le n$.
		
		\item[\rm (iii)] If $m\ge n+3$ and $R_{n+3}-R_{n+2}\le 2e$, then Theorem \ref{thm:beligeneral}(iv) holds for $2\le i\le n+1$.
	\end{enumerate} 
\end{cor}
\begin{proof}
	It is clear from Lemma \ref{lem:B1-B3-B4-j}(v).
\end{proof}
\begin{lem}\label{lem:m=n+1representation}
	Let $ N $ be an $ \mathcal{O}_{F} $-lattice of rank $ n\ge 2 $ and $ m=n+1 $. Suppose that $n$ is odd and $M$ satisfies the conditions: $R_{n-1}=R_{n}=0$ and
	\begin{align*}
		\text{either}\quad
		\begin{cases}
			R_{n+1}=0,\,\alpha_{n}=1   \\
			d((-1)^{(n+1)/2}a_{1,n+1})=1
		\end{cases} 
	\text{or}\quad R_{n+1}=1.
	\end{align*}	 
	If $ FM $ represents $ FN $, then $ M $ represents $ N $.
\end{lem}
\begin{proof}
	 By Proposition \ref{prop:Ralphaproperty3}(v), we have $R_{i}=0$ for $1\le i\le n$. If $R_{n+1}=1$, then $\alpha_{n}=1$ by Proposition \ref{prop:alphaproperty}(iii). Since $\ord(a_{1,n+1})$ is odd, $d((-1)^{(n+1)/2}a_{1,n+1})=0=1-R_{n+1}$. Hence, in both cases, we have $\alpha_{n}=1$ and $R_{n+1}+d[(-1)^{(n+1)/2}a_{1,n+1}]=R_{n+1}+d((-1)^{(n+1)/2}a_{1,n+1})=1 $. Also, Proposition \ref{prop:Ralphaproperty3}(vi) implies that $\alpha_{i}=1$ for $1\le i\le n-1$. Therefore, we are done by combining Corollaries \ref{cor:B1}(ii), \ref{cor:B2}(iii), \ref{cor:B3}(iii), \ref{cor:B4}(i) and Theorem \ref{thm:beligeneral}.
\end{proof}
\begin{lem}\label{lem:m=n+2representation}
	Let $ N $ be an $ \mathcal{O}_{F} $-lattice of rank $ n\ge 2 $ and $ m=n+2 $. 
	\begin{enumerate}
		\item[\rm (i)]$n$ is even and $ M $ satisfies the conditions: $R_{n}=R_{n+1}=0$ and
		\begin{align*}
			\text{either}\quad
			\begin{cases}
				R_{n+2}=0,\,\alpha_{n+1}=1   \\
				d((-1)^{(n+2)/2}a_{1,n+2})=1
			\end{cases} 
			\text{or}\quad R_{n+2}=1.
		\end{align*} 
		If $ FM $ represents $ FN $, then $ M $ represents $ N $.
		
		\item[\rm (ii)] $n$ is odd and $M$ satisfies the conditions: $R_{n}=R_{n+1}=0$ and
		\begin{align*}
			\text{either}\quad
			\begin{cases}
				R_{n+2}=0\\
				\alpha_{n}=1 			 
			\end{cases} 
			\text{or}\quad R_{n+2}=1.
		\end{align*} 
		If $ FM $ represents $ FN $, then $ M $ represents $ N $.
	\end{enumerate} 
\end{lem}
\begin{proof}
	By Proposition \ref{prop:Ralphaproperty3}(v), we have $R_{i}=0$ for $1\le i\le n+1$ for even or odd $n$. 
	
	(i) If $R_{n+2}=1$, then $\alpha_{n+1}=1$ by Proposition \ref{prop:alphaproperty}(iii). Since $\ord(a_{1,n+2})$ is odd, we have $d((-1)^{(n+2)/2}a_{1,n+2})=0=1-R_{n+2}$. Hence, in both cases, we see that $ \alpha_{n+1}=1$ and $R_{n+2}+d[(-1)^{(n+2)/2}a_{1,n+2}]=R_{n+2}+d((-1)^{(n+2)/2}a_{1,n+2})=1 $. Also, Proposition \ref{prop:Ralphaproperty3}(vi) implies that $\alpha_{i}=1$ for $1\le i\le n$. So we are done by combining Corollaries \ref{cor:B1}(i), \ref{cor:B2}(ii), \ref{cor:B3}(ii), \ref{cor:B4}(ii) and Theorem \ref{thm:beligeneral}.
	
	(ii) Since $\ord(a_{1,n+1})$ is even, we have $d((-1)^{(n+1)/2}a_{1,n+1})\ge 1$.
	 
	If $R_{n+2}=0$ and $\alpha_{n}=1$, then $\alpha_{n+1}=1$ by Proposition \ref{prop:Rproperty}(vi). If $R_{n+2}=1$, then $\alpha_{n+1}=1$ by Proposition \ref{prop:alphaproperty}(iii). So, in both cases, Proposition \ref{prop:Ralphaproperty3}(vi) implies that $\alpha_{i}=1$ for $1\le i\le n$.
	Hence
	 	\begin{align*}
	 		d[(-1)^{(n+1)/2}a_{1,n+1}]=\min\{d((-1)^{(n+1)/2}a_{1,n+1}),\alpha_{n+1}\}=\alpha_{n+1}=1=1-R_{n+1}.
	 	\end{align*}
Combining Corollaries \ref{cor:B1}(ii) \ref{cor:B2}(iii) \ref{cor:B3}(iii), \ref{cor:B4}(ii) and Theorem \ref{thm:beligeneral}, we are done.
\end{proof} 
 
\section{Classic $ n $-universality conditions for even $ n $} \label{sec:classicaleven-con}
Throughout this section, we assume that $ n $ is an even integer and $ m\ge n+2\ge 4 $. 
\begin{thm}\label{thm:classicaleven-nuniversaldyadic}
	 $ M $ is n-universal if and only if $ FM $ is n-universal and $ M $ satisfies the following conditions:
	 
	 $ J_{1}^{E}(n) $: $ R_{i}=0 $ for $ i=1,\ldots,n+1 $.
	 
	 $ J_{2}^{E}(n) $: $ \alpha_{n+1}=1 $ and $ R_{n+2}+d[(-1)^{(n+2)/2}a_{1,n+2}]=1 $. Also, if $ n=2 $, then $ m\ge  5 $.
	 
	 $ J_{3}^{E}(n) $: $ R_{n+3}-R_{n+2}\le 2e $.
\end{thm}
\begin{lem}\label{lem:classicalcon12even}
	 Suppose that $ FM $ is n-universal. Then the following conditions are equivalent.
	 \begin{enumerate}
	 	 \item[\rm (i)] Theorem \ref{thm:beligeneral}(i)(ii) hold for all n-ary $ \mathcal{O}_{F} $-lattices $ N $.
	 	 
	 	 \item[\rm (ii)] Theorem \ref{thm:beligeneral}(i)(ii) holds for $ N=H_{e}^{n}(1),C_{1}^{n}(\omega) $ (cf. Definition \ref{defn:classicallattices}). 
	 	 
	 	 \item[\rm (iii)] $ M $ satisfies the following conditions:
	 	 
	 	 $ J_{1}^{\prime E}(n) $: $ R_{i}=0 $ for $ i=1,\ldots,n+1 $ and $\alpha_{i}=1 $ for $ i=1,\ldots,n $. 
	 	 
	 	 $ J_{2}^{\prime E}(n) $: If $ e>1 $, then $ d[(-1)^{(n+2)/2}a_{1,n+2}]\le 1-R_{n+2} $.
	 \end{enumerate} 
\end{lem}
	
\begin{proof}
	\textbf{(i)$ \Rightarrow $(ii):} It is trivial.
	
	\textbf{(ii)$ \Rightarrow $(iii):}  We are going to show that $M$ satisfies $ J_{1}^{\prime E}(n) $ and $ J_{2}^{\prime E}(n) $ by proving the assertions (a)-(c).
	
	\textbf{(a) $ R_{i}=0 $ for $ i=1,\ldots,n+1 $.}
	
	Take $N=C_{1}^{n}(\omega)$. Then $S_{i}=0$ for $1\le i\le n$ by Lemma \ref{lem:S-invariant}(i). Applying Theorem \ref{thm:beligeneral}(i), we have $R_{1}\le S_{1}=0 $. Hence $ R_{1}=0 $ by \eqref{eq:integralcondition} and thus $R_{2}\ge 0$ by \eqref{eq:classicallyintegral}.  
   
   Take $N=H_{e}^{n}(1)$. Then $S_{n-1}=e$ and $ S_{n}=-e$ by Lemma \ref{lem:S-invariant}(i). Applying Theorem \ref{thm:beligeneral}(i), since $ 1<n<m $ and $ R_{n}\ge R_{2}\ge 0>-e=S_{n}  $ (by Proposition \ref{prop:Ralphaproperty3}(i)), we have 
    $R_{n}+R_{n+1}\le S_{n-1} +S_{n} =0$. Hence Proposition \ref{prop:Ralphaproperty3}(iv) implies $ R_{i}=0 $ for $1\le i\le n+1$.
	
	\textbf{(b) $ \alpha_{i}=1 $ for $ i=1,\ldots,n $.}
	
	Take $N=C_{1}^{n}(\omega)$. Then $\beta_{1}=1$ by Lemma \ref{lem:S-invariant}(iii). Applying Theorem \ref{thm:beligeneral}(ii) with $i=1$, we have
	\begin{align*}
	\min\{e,d[-a_{1,2}]\}=A_{1}(M,C_{1}^{n}(\omega))\le d[a_{1}b_{1}]\le \beta_{1}=1.
	\end{align*}
	By (a), we have $R_{1}=R_{2}=0$. Hence from \eqref{eq:alpha-defn} we deduce $\alpha_{1}=\min\{e,d[-a_{1,2}]\}\le 1$. So (b) follows by Proposition \ref{prop:Ralphaproperty3}(vi).
	
	\textbf{(c) If $ e>1 $, then  $ d[(-1)^{(n+2)/2}a_{1,n+2}]\le 1-R_{n+2} $.}
	
	Take $N=H_{e}^{n}(1)$. Since $S_{n}-S_{n-1}=-2e$, we have $d[-a_{1,n+1}b_{1,n-1}]=\beta_{n-1}=0$ by Proposition  \ref{prop:alphaproperty}(ii). Since $S_{n-1}-S_{n-2}=2e$, $ \beta_{n-2}=2e $ by Proposition \ref{prop:Rproperty}(i). Hence
	 \begin{align*}
	 	&d[a_{1,n+2}b_{1,n-2}]\\
	 	=\;&
	 	\begin{cases}
	 		\min\{d((-1)^{(n+2)/2}a_{1,n+2}),\alpha_{n+2},2e\}=\min\{d[(-1)^{(n+2)/2}a_{1,n+2}],2e\} &\text{if $ n>2 $},  \\
	 		\min\{d((-1)^{(n+2)/2}a_{1,n+2}),\alpha_{n+2}\}=	\min\{d[(-1)^{(n+2)/2}a_{1,n+2}]\} &\text{if $ n=2 $}. 
	 	\end{cases} 
	 \end{align*}
	Apply Theorem \ref{thm:beligeneral}(ii) with $i=n$, we have	 
	\begin{align*}
	\min\{e, R_{n+2}+d[(-1)^{(n+2)/2}a_{1,n+2}],R_{n+2}+2e\}&=\min\{\dfrac{e}{2}+e, e, R_{n+2}+d[a_{1,n+2}b_{1,n-2}]\} \\
	&=A_{n}(M,H_{e}^{n}(1))\le  d[a_{1,n}b_{1,n} ]\le \alpha_{n}\underset{(b)}{=}1,
	\end{align*}
    where the term $ R_{n+2}+2e $ is ignored if $ n=2 $. Since $ R_{n+2}\ge R_{2}=0 $ by Proposition \ref{prop:Ralphaproperty3}(i), $R_{n+2}+2e>1$. If $e>1$, then $ R_{n+2}+d[(-1)^{(n+2)/2}a_{1,n+2}]=A_{n}(M,H_{e}^{n}(1))\le 1 $, as required.
	
	\textbf{(iii)$ \Rightarrow $(i):} This follows by Corollaries \ref{cor:B1}(i) and \ref{cor:B2}(ii).
\end{proof}
\begin{lem}\label{lem:classicalevenlatticeproperty2}
	 Suppose that $ M $ satisfies $ J_{1}^{\prime E}(n) $. Suppose $ e=1 $ and $ R_{n+2}+d[(-1)^{(n+2)/2}a_{1,n+2}]>1 $.
	\begin{enumerate}
		\item[\rm (i)] If either $ d((-1)^{(n+2)/2}a_{1,n+2})<2e $, or $ d[(-1)^{(n+2)/2}a_{1,n+2}]\ge 2e$,  then Theorem \ref{thm:beligeneral}(iii) fails at $i=n+1$ for either $N=H_{e}^{n}(1)$ or $N=H_{e}^{n}(\Delta)$.
		
		\item[\rm (ii)] Assume that $m\ge n+3$ and  $d[(-1)^{(n+2)/2}a_{1,n+2}]=\alpha_{n+2}$. If $ R_{n+2}\ge 1 $, then Theorem \ref{thm:beligeneral}(iii) fails at $i=n+1$ for either $  N=C_{1}^{n}(\omega) $ or $  N=C_{2}^{n}(\omega) $.
	\end{enumerate}
\end{lem}
\begin{proof}
	(i) Take $N=H_{e}^{n}(\mu)$ with $\mu\in \{1,\Delta\}$. Since $e=1$, we have $S_{n}=-e=-1$ by Lemma \ref{lem:S-invariant}(i). Hence $ R_{n+2}\ge R_{n}=0>S_{n}=-1 $ by Proposition \ref{prop:Ralphaproperty3}(i). Since $S_{n}=-e$, Proposition \ref{prop:Ralphaproperty3}(iii) implies
	\begin{align*}
		d[(-1)^{n/2}b_{1,n}]\ge 2e.
	\end{align*}
	 If $ d((-1)^{(n+2)/2}a_{1,n+2})<2e  $, then $d[(-1)^{(n+2)/2}a_{1,n+2}]\le d((-1)^{(n+2)/2}a_{1,n+2})<2e  $ and so
	$ d[-a_{1,n+2}b_{1,n}] =d[(-1)^{(n+2)/2}a_{1,n+2}]>1-R_{n+2}  $ by the domination principle. Hence
	\begin{align*}
		d[-a_{1,n+1}b_{1,n-1}]+d[-a_{1,n+2}b_{1,n}]>0+(1-R_{n+2})=2e-1-R_{n+2}=2e+S_{n}-R_{n+2}.
	\end{align*}
	If $d[(-1)^{(n+2)/2}a_{1,n+2}]\ge 2e $, then $d[-a_{1,n+2}b_{1,n}]\ge 2e$ by the domination principle. Hence
	\begin{align*}
		d[-a_{1,n+1}b_{1,n-1}]+d[-a_{1,n+2}b_{1,n}]\ge 0+2e>  2e-1-R_{n+2}=2e+S_{n}-R_{n+2}.
	\end{align*}
	For the second part, by definition, $[b_{1},\ldots,b_{n}]=FN=FH_{1}^{n}(1)\cong \mathbb{H}^{n/2}$ or $ =FH_{1}^{n}(\Delta)\cong \mathbb{H}^{(n-2)/2}\perp [\pi,-\Delta\pi] $ (as $e=1$). Hence, by Lemma \ref{lem:spacerep-criterion}(i), $[a_{1},\ldots,a_{n+1}]$ does not represent  $ FH_{1}^{n}(1) $ or $ FH_{1}^{n}(\Delta)$.
	
	(ii) Take $N=C_{\nu}^{n}(\omega)$, where $\nu\in \{1,2\}$. Then $S_{i}=0$ for $1\le i\le n$ by Lemma \ref{lem:S-invariant}(i). Clearly, $ R_{n+2}\ge 1>S_{n}=0 $ from the hypothesis. Since $\ord(a_{1,n+1}b_{1,n-1})$ is even, $ d(-a_{1,n+1}b_{1,n-1})\ge 1 $. Also since $R_{n+2}-R_{n+1}=R_{n+2}\ge 1>-2e$, Proposition \ref{prop:alphaproperty}(ii) implies $ \alpha_{n+1}\ge 1 $. We have assumed $e=1$. Since $ S_{n}-S_{n-1}=0  $, Proposition \ref{prop:alphaproperty}(iii) implies $ \beta_{n-1}=1 $. Hence 
	\begin{align*}
		d[-a_{1,n+1}b_{1,n-1} ]=\min\{d(-a_{1,n+1}b_{1,n-1}),\alpha_{n+1},\beta_{n-1}\}=\beta_{n-1}=1.
	\end{align*}
 Next, we estimate the term $d[-a_{1,n+2}b_{1,n}]$. First, we have $d[(-1)^{(n+2)/2}a_{1,n+2}]=\alpha_{n+2}$ from the hypothesis. Also, we have $d[(-1)^{n/2}b_{1,n}]=d((-1)^{n/2}b_{1,n})=d(\omega)=1$ by definition of $C_{\nu}^{n}(\omega)$. By the domination principle, we see that
	\begin{align*}
			d[-a_{1,n+2}b_{1,n}]\ge \min\{d[(-1)^{(n+2)/2}a_{1,n+2}],d[(-1)^{n/2}b_{1,n}]\}=\min\{\alpha_{n+2},1\}.
	\end{align*}
 If $\alpha_{n+2}\ge 1$, then $d[-a_{1,n+2}b_{1,n}]+R_{n+2}\ge 1+R_{n+2}>1 $ 
from the hypothesis  $R_{n+2}>0$. If $ \alpha_{n+2}=0 $,  then $ R_{n+2}=R_{n+3}+2e\ge 2e  $ by Propositions \ref{prop:alphaproperty}(ii) and \ref{prop:Ralphaproperty3}(i). It follows that $ d[-a_{1,n+2}b_{1,n}]+R_{n+2}\ge 0+2e=2e>1 $. Hence, in both cases, we have
	\begin{align*}
		d[-a_{1,n+2}b_{1,n}]>1-R_{n+2}.
	\end{align*}
 So we conclude that
	\begin{align*}
		d[-a_{1,n+1}b_{1,n-1}]+d[-a_{1,n+2}b_{1,n}]>1+(1-R_{n+2})=2e+S_{n}-R_{n+2}.
	\end{align*}
	For the second part, by definition, $ [b_{1},\ldots,b_{n}]=FN=FC_{i}^{n}(\omega)\cong W_{i}^{n}(\omega) $ with $i\in \{1,2\}$ (cf. \cite[Definition 3.4]{HeHu2}), so $[a_{1},\ldots,a_{n+2}]$ cannot represent both of $FC_{1}^{n}(\omega)$ and $FC_{2}^{n}(\omega)$ by \cite[Lemma 3.13]{HeHu2}, as desired.
\end{proof}
\begin{lem}\label{lem:0-1}
	 Suppose that $M$ satisfies $R_{j-2}=0$  and $ R_{j}+d[(-1)^{j/2}a_{1,j}]\le 1$ for some even $j\ge 4$. Then $\{R_{j}, d[(-1)^{j/2}a_{1,j}]\}\subseteq \{0,1\}$.
\end{lem}
\begin{proof}
	By Proposition \ref{prop:Ralphaproperty3}(i), we have $R_{j}\ge R_{j-2}=0$. Since also $d[(-1)^{j/2}a_{1,j}]\ge 0$, we see that $ R_{j}\le 1$ and $ d[(-1)^{j/2}a_{1,j}] \le 1$ from the inequality we assumed, as desired.
\end{proof}
\begin{lem}\label{lem:classicalcon3even}
	 Suppose that $ FM $ is n-universal and $ M $ satisfies $ J_{1}^{\prime E}(n) $ and $ J_{2}^{\prime E}(n) $ in Lemma \ref{lem:classicalcon12even}. Then the following conditions are equivalent.
	 \begin{enumerate}
	 	\item[\rm (i)] Theorem \ref{thm:beligeneral}(iii) holds for n-ary $ \mathcal{O}_{F} $-lattices $ N $.
	 	
	 	\item[\rm (ii)] Theorem \ref{thm:beligeneral}(iii) holds for $ N=H_{e}^{n}(1),C_{1}^{n}(\omega) \;\text{and}\;C_{2}^{n}(\omega) $, and for $ N=H_{e}^{n}(\Delta) $ if $ e=1 $ (cf. Definition \ref{defn:classicallattices}).
	 	
	 	\item[\rm (iii)] $ M $ satisfies $ J_{2}^{E}(n) $ in Theorem \ref{thm:classicaleven-nuniversaldyadic}.
	 \end{enumerate}
\end{lem}
\begin{proof}
	\textbf{(i)$ \Rightarrow $(ii):} It is trivial.
	
	\textbf{(ii)$ \Rightarrow $(iii):} Firstly, we claim 
	\begin{align}\label{eq:da1n+2}
		d[(-1)^{(n+2)/2}a_{1,n+2}]\le 1-R_{n+2}.
	\end{align}
	 Since $ M $ satisfies $ J_{2}^{\prime E}(n) $, \eqref{eq:da1n+2} is clear for $ e>1 $. Assume $ e=1 $ and $ d[(-1)^{(n+2)/2}a_{1,n+2}]>1-R_{n+2} $. Then we show the claim by proving the assertions (a) and (b).
	
	\textbf{(a) $ d[(-1)^{(n+2)/2}a_{1,n+2}]=\alpha_{n+2}<2e$.}
	
	If $ d((-1)^{(n+2)/2}a_{1,n+2})<2e $ or $d[(-1)^{(n+2)/2}a_{1,n+2}]\ge 2e$, then Theorem \ref{thm:beligeneral}(iii) fails at $i=n+1$ for either $N=H_{1}^{n}(1)$ or $N=H_{1}^{n}(\Delta)$ by Lemma \ref{lem:classicalevenlatticeproperty2}(i), a contradiction. Hence $ d((-1)^{(n+2)/2}a_{1,n+2})\ge 2e>d[(-1)^{(n+2)/2}a_{1,n+2}] $. So we must have $m\ge n+3$ and $  d[(-1)^{(n+2)/2}a_{1,n+2}]=\alpha_{n+2}<2e $.
	
	\textbf{(b) $ R_{n+2}=0 $.}
	
	By Proposition \ref{prop:Ralphaproperty3}(i), we have $ R_{n+2}\ge R_{2}=0 $. Suppose $ R_{n+2}\ge 1 $. This combined with (a) shows that Theorem \ref{thm:beligeneral}(iii) fails at $i=n+1$ for either $N=C_{1}^{n}(\omega)$ or $N=C_{2}^{n}(\omega)$ by Lemma \ref{lem:classicalevenlatticeproperty2}(ii). This is a contradiction. (b) is proved.
	
	Now, combining the assumption, (a) and (b), we conclude that
	 \begin{align}\label{eq:R4alpha4}
		1<R_{n+2}+d[(-1)^{(n+2)/2}a_{1,n+2}]=0+\alpha_{n+2}=\alpha_{n+2}<2e.
	\end{align}
 	Note that $ \alpha_{n+2}\in \mathbb{Z} $ by Proposition \ref{prop:alphaproperty}(i). This implies from \eqref{eq:R4alpha4} that $2\le \alpha_{n+2}<2e$, i.e. $ e>1 $. This contradicts the assumption $ e=1 $ and thus the claim is proved. 
	
	\medskip
	Next, we are going to show that $M$ satisfies $J_{2}^{E}(n)$. By the claim \eqref{eq:da1n+2} and Lemma \ref{lem:0-1} with $j=n+2\ge 4$, we have $R_{n+2}\in \{0,1\}$. 
	
	 To show $ \alpha_{n+1}=1 $, by $J_{1}^{\prime E}(n)$, we have $R_{i}=0$ for $1\le i\le n+1$. If $R_{n+2}=1$, then $ \alpha_{n+1}=1 $ by Proposition \ref{prop:alphaproperty}(iii). If $R_{n+2}=0$, since $\alpha_{n}=1$, Proposition \ref{prop:Ralphaproperty3}(vi) implies $ \alpha_{n+1}=1 $.
	
	By the claim, it remains to show $d[(-1)^{(n+2)/2}a_{1,n+2}]\ge 1-R_{n+2}$ for the second equality in $J_{2}^{E}(n)$. For $1\le i\le n+1$, since $\alpha_{i}=1$,  Proposition \ref{prop:alphaproperty}(v) implies $d[-a_{i}a_{i+1}]\ge 1+R_{i}-R_{i+1}\ge 1-R_{n+2}$. Hence $d[(-1)^{(n+2)/2}a_{1,n+2}]\ge 1-R_{n+2}$ by the domination principle, as required.
	
 	Finally, if $n=2$, since $ FM $ is $2$-universal, we have either $ m\ge 5 $, or $ m=4 $ and $ FM\cong \mathbb{H}^{2} $ by \cite[Theorem 2.3]{hhx_indefinite_2021}. If $ m=4 $, then $ d[a_{1,4}]=d(a_{1,4})=\infty $, which contradicts $ R_{4}+d[a_{1,4}]=1$. Thus we must have $ m\ge 5 $.
	
	\textbf{(iii)$ \Rightarrow $(i):} This follows by Corollary \ref{cor:B3}(ii).
\end{proof}
\begin{lem}\label{lem:classicalcon4even}
	 Suppose that $ FM $ is n-universal and $ M $ satisfies $ J_{1}^{\prime E}(n) $ in Lemma \ref{lem:classicalcon12even} and $ J_{2}^{E}(n) $ in Theorem \ref{thm:classicaleven-nuniversaldyadic}. Then the following conditions are equivalent.
	\begin{enumerate}
		\item[\rm (i)] Theorem \ref{thm:beligeneral}(iv) holds for all n-ary $ \mathcal{O}_{F} $-lattices $ N $. 
		
		\item[\rm (ii)] Theorem \ref{thm:beligeneral}(iv) holds for all the lattices $N$ in the following list if $ m\ge n+3 $ and $ R_{n+3}-R_{n+2}>2e $:
		\begin{align*}
			 C_{1}^{n}(c)\;\;\text{and}\;\;C_{2}^{n}(c)\;\;\text{with}\;\;c\in F^{\times}/F^{\times 2}\;\;\text{and}\;\;  d(c)\in \{0,1\}  
		\end{align*}
		(cf. Definition \ref{defn:classicallattices}).
		
		\item[\rm (iii)] $ M $ satisfies $ J_{3}^{E}(n) $ in Theorem \ref{thm:classicaleven-nuniversaldyadic}.
	\end{enumerate} 
\end{lem}
\begin{proof}
	\textbf{(i)$ \Rightarrow $(ii):} It is trivial.
	
	\textbf{(ii)$ \Rightarrow $(iii):} Assume $ R_{n+3}-R_{n+2}>2e $. Then $ \alpha_{n+2}>2e>d[(-1)^{(n+2)/2}a_{1,n+2}]$ by Proposition \ref{prop:Rproperty}(i) and Lemma \ref{lem:0-1}. Hence  $d((-1)^{(n+2)/2}a_{1,n+2})=d[(-1)^{(n+2)/2}a_{1,n+2}] \in \{0,1\}$ by Lemma \ref{lem:0-1}.
	
	Write $V:=[a_{1},\ldots,a_{n+2}]$. Let $N=C_{\nu}^{n}(c)$, with $\nu\in\{1,2\}$ and $c=(-1)^{(n+2)/2}a_{1,n+2}\in F^{\times}/F^{\times 2}$. Then $d(c)\in \{0,1\}$. Now $\det V=a_{1,n+2}=(-1)^{(n+2)/2}c=-\det FN$. Since $S_{n}=1-d(c)=R_{n+2}$, we have $R_{n+3}>S_{n}+2e=R_{n+2}+2e$, so $FN=[b_{1},\ldots,b_{n}]$ is represented by $V=[a_{1},\ldots,a_{n+2}]$ by condition (ii). Hence $V$ represents both $FC_{1}^{n}(c)\cong W_{1}^{n}(c)$ and $ FC_{2}^{n}(c)\cong W_{2}^{n}(c)$, which contradicts \cite[Lemma 3.13]{HeHu2}. Thus $R_{n+3}-R_{n+2}\le 2e$.
	
	\textbf{(iii)$ \Rightarrow $(i):} This follows by Corollary \ref{cor:B4}(iii).
\end{proof}
\begin{proof}[Proof of Theorem \ref{thm:classicaleven-nuniversaldyadic}]
We claim the following equivalence holds:
	  \begin{align}\label{equi:even}
		J_{1}^{\prime E}(n),\;J_{2}^{\prime E}(n)\;\text{and}\; J_{2}^{E}(n)\Longleftrightarrow \;	J_{1}^{ E}(n)\;\text{and}\; J_{2}^{E}(n).
	\end{align}
	Necessity is clear. For sufficiency, since $R_{i}=0$ for $1\le i\le n+1$ and $\alpha_{n+1}=1$, Proposition \ref{prop:Ralphaproperty3}(vi) implies that $\alpha_{i}=1$ for $1\le i\le n$. Thus the claim is proved. Now the theorem follows immediately by the claim, Lemmas \ref{lem:classicalcon12even}, \ref{lem:classicalcon3even}, \ref{lem:classicalcon4even} and Theorem \ref{thm:beligeneral}.
\end{proof}

\section{Classic $ n $-universality conditions for odd $ n $}\label{sec:classicalodd-con}
Throughout this section, we assume that $ n $ is an odd integer and $ m\ge n+2\ge 5 $. 
\begin{thm}\label{thm:classicalodd-nuniversaldyadic}
	$ M $ is n-universal if and only if $m\ge n+3$ and $ M $ satisfies  the following conditions:
	
 	$ J_{1}^{O}(n) $: $R_{i}=0$ for $1\le i\le n$, $\alpha_{n}=1$ and $R_{n+1}+d[(-1)^{(n+1)/2}a_{1,n+1}]=1$.
 
 	 $ J_{2}^{O}(n) $: If either $ R_{n+1}=1 $ or $ R_{n+2}>1 $, then $ \alpha_{n+2}\le G_{n} $, where
	 \begin{equation}\label{eq5.1new}
 		\begin{split}
 			G_n:&= 2(e-\lfloor(R_{n+2}-R_{n+1})/2\rfloor)-1\\
 			&=\begin{cases}
 			2e-R_{n+2}+R_{n+1}-1 &\text{if $ R_{n+2}-R_{n+1} $ is even}, \\
 			2e-R_{n+2}+R_{n+1} &\text{if $ R_{n+2}-R_{n+1} $ is odd.}
 			\end{cases}
 		\end{split}
 	\end{equation}
	
	 $ J_{3}^{O}(n) $: $ R_{n+3}-R_{n+2}\le 2e $.	 	
\end{thm}
\begin{proof}
	First, $J_{1}^{O}(n)$ is the same as $J_{1}^{E}(n-1)$ and $J_{2}^{E}(n-1)$. Then, since the conditions $J_{2}^{O}(n)$ and $\alpha_{n}=1$ hold, Lemma \ref{lem:simplifythm} below implies $J_{3}^{E}(n-1)$. Hence the theorem is equivalent to Proposition \ref{prop:classicalodd-nuniversaldyadic}. 
	
	We will show a series of lemmas to complete the proof of Proposition \ref{prop:classicalodd-nuniversaldyadic} in the coming discussion.
\end{proof}
\begin{prop}\label{prop:classicalodd-nuniversaldyadic}
	$ M $ is n-universal if and only if $ FM $ is n-universal and $ M $ satisfies $ J_{1}^{E}(n-1) $, $ J_{2}^{E}(n-1) $, $ J_{3}^{E}(n-1) $ in Theorem \ref{thm:classicaleven-nuniversaldyadic} and $J_{2}^{O}(n)$ and $J_{3}^{O}(n)$ in Theorem \ref{thm:classicalodd-nuniversaldyadic}.
\end{prop}
\begin{lem}\label{lem:simplifythm}
	Suppose that $ M $ satisfies $ J_{2}^{O}(n) $. If $\alpha_{n}=1$, then $R_{n+2}-R_{n+1}\le 2e-1$, i.e. $J_{3}^{E}(n-1)$; if moreover either $R_{n+1}=1$ or $ R_{n+2}>1$, then  $ R_{n+3}-R_{n+2}\le 2e-1$.
\end{lem}
\begin{proof}
   Since $\alpha_{n}=1$, we have $R_{n+1}\in \{0,1\}$ by Proposition \ref{prop:alphaproperty}(iii). 
   
    If $ R_{n+2}-R_{n+1}\ge 2e $, then $R_{n+2}\ge R_{n+1}+2e\ge 2e>1$. Hence
 	\begin{align*}
 		\alpha_{n+2}\le G_{n}=2(e-\lfloor(R_{n+2}-R_{n+1})/2\rfloor)-1\le 2(e-\lfloor 2e/2\rfloor)-1=-1
 	\end{align*}
 	by $J_{2}^{O}(n)$. This contradicts Proposition \ref{prop:alphaproperty}(i) and so $ R_{n+2}-R_{n+1}\le 2e-1 $. 

 	If either $R_{n+1}=1$ or $R_{n+2}>1$, from  \cite[Remark 5.2]{HeHu2} we have $ R_{n+2}\ge 1 $. Since $R_{n+1}\in \{0,1\}$, we further have $R_{n+2}-R_{n+1}\ge 0$. Hence
 	\begin{align*}
 		\alpha_{n+2}\le G_{n}=2(e-\lfloor(R_{n+2}-R_{n+1})/2\rfloor)-1\le 2(e-\lfloor 0/2\rfloor)-1=2e-1
 	\end{align*}
 	by $J_{2}^{O}(n)$. So Proposition \ref{prop:Rproperty}(i) implies $ R_{n+3}-R_{n+2}\le 2e-1 $.
\end{proof}
\begin{lem}\label{lem:classicalcon12odd}
	 Suppose that $ FM $ is n-universal. The following conditions are equivalent.
	\begin{enumerate}
		\item[\rm (i)] Theorem \ref{thm:beligeneral}(i)(ii) hold for all n-ary $ \mathcal{O}_{F} $-lattices $ N $.
		
		\item[\rm (ii)] Theorem \ref{thm:beligeneral}(i)(ii) hold for $ N=C_{1}^{n}(\omega)  $ (cf. Definition \ref{defn:classicallattices}).
		
		\item[\rm (iii)] $ M $ satisfies $ J_{1}^{\prime E}(n-1) $ in Lemma \ref{lem:classicalcon12even}.
	\end{enumerate}
\end{lem}
\begin{proof}
	\textbf{(i)$ \Rightarrow $(ii)}: It is trivial.
	
	\textbf{(ii)$ \Rightarrow $(iii)}:  Take $N=C_{1}^{n}(\omega)$. Then $S_{i}=0$ for $1\le i\le n$ by Lemma \ref{lem:S-invariant}(ii). Applying Theorem \ref{thm:beligeneral}(i), we have $R_{1}\le S_{1}=0 $. Hence $ R_{1}=0 $ by \eqref{eq:integralcondition}. By \cite[Lemma 4.6(i)]{beli_representations_2006}, we also have $R_{n-1}+R_{n}\le  S_{n-1} +S_{n} =0$. Proposition \ref{prop:Ralphaproperty3}(iv) implies $ R_{i}=0 $ for $ 1\le i\le n $.
	
	By Lemma \ref{lem:S-invariant}(iii), we have $\beta_{1}=1$. Applying Theorem \ref{thm:beligeneral}(ii) with $i=1$, we have 
	\begin{align*}
		\min\{e,d[-a_{1,2}]\}=A_{1}(M,C_{1}^{n}(\omega))\le d[a_{1}b_{1}]\le \beta_{1}=1.
	\end{align*}
	Since $R_{2}-R_{1}=0$, it follows that $ \alpha_{1}=\min\{e,d[-a_{1,2}]\}\le 1 $ by \eqref{eq:alpha-defn}. So $ \alpha_{i}=1 $ for $1\le i\le n-1$ by Proposition \ref{prop:Ralphaproperty3}(vi).
	
	\textbf{(iii)$ \Rightarrow $(i)}: It is straightforward from Corollaries \ref{cor:B1}(ii) and \ref{cor:B2}(iii).
\end{proof}
\begin{lem}\label{lem:oddlatticeproperty2-1}
	Suppose that $ M $ satisfies $J_{1}^{\prime E}(n-1)$ and $ J_{2}^{E}(n-1)$. If either $ R_{n+1}=1 $ or $ R_{n+2}>1 $, then $ d((-1)^{(n+1)/2}a_{1,n+1})=1-R_{n+1} $, $ ((-1)^{(n+1)/2}a_{1,n+1})^{\#} $ is a unit and $ d(((-1)^{(n+1)/2}a_{1,n+1})^{\#})=2e+R_{n+1}-1 $.
\end{lem}
\begin{proof}
	Since $R_{n-1}=0$, applying Lemma \ref{lem:0-1} with $j=n+1$, we have $R_{n+1}\in \{0,1\}$. Then the argument is similar to \cite[Lemma 5.8]{HeHu2}.
\end{proof}
\begin{lem}\label{lem:oddlatticeproperty2-2}
	Suppose that $ M $ satisfies $J_{1}^{\prime E}(n-1)$ and $J_{2}^{E}(n-1)$. Assume that $ \alpha_{n+2}>G_{n}$ (cf. $\eqref{eq5.1new}$) and either $ R_{n+1}=1 $ or $ R_{n+2}>1 $. Let $c=(-1)^{(n+1)/2}a_{1,n+2}$ and $\tilde{c}=(-1)^{(n+1)/2}a_{1,n+1} $.
	\begin{enumerate}
		\item[\rm (i)] We have $R_{n+2}>S_n$ and $d[-a_{1,n+1}b_{1,n-1}]+d[-a_{1,n+2}b_{1,n}]>2e+S_n-R_{n+2}$ for both $N=C_{1}^{n}(c)$ and $N=C_{1}^{n}(c\tilde{c}^{\#})$.
		
		\item[\rm (ii)] $[a_1,\ldots, a_{n+1}]$ does not represent $FN=[b_1,\ldots, b_n]$ for $N=C^n_1(c)$ or $N=C^n_1(c\tilde{c}^{\#})$.
	\end{enumerate}
	
 	Thus Theorem \ref{thm:beligeneral}(iii) fails at $ i=n+1 $ for at least one of the lattices  $C_{1}^{n}(c) $ and $ C_{1}^{n}(c\tilde{c}^{\#}) $.
\end{lem}
\begin{proof}
	(i) Following the argument in \cite[Lemma 5.9(i)]{HeHu2}, we only need to show
	 $d[-a_{1,n+1}b_{1,n-1}]=1-R_{n+1}$. By Lemma \ref{lem:S-invariant}(ii), we have $S_{n}-S_{n-1}\not=-2e$, so Proposition \ref{prop:Rproperty}(ii) implies $\beta_{n-1}\ge 1$. This combined with $d((-1)^{(n-1)/2}b_{1,n-1})=\infty$ shows that $d[(-1)^{(n-1)/2}b_{1,n-1}]\ge 1$.
	
	By $J_{2}^{E}(n-1)$ and Lemma \ref{lem:0-1}, we have $d[(-1)^{(n+1)/2}a_{1,n+1}]=1-R_{n+1}$ and $ d[(-1)^{(n+1)/2}a_{1,n+1}]\in \{0,1\}$. If $ d[(-1)^{(n+1)/2}a_{1,n+1}]=0$, then $R_{n+1}=1$ and so $d[-a_{1,n+1}b_{1,n-1}]=0=1-R_{n+1}$ by the domination principle. If $ d[(-1)^{(n+1)/2}a_{1,n+1}]=1$, then $R_{n+1}=0$. By Proposition \ref{prop:Ralphaproperty3}(i), we have $R_{n+2}-R_{n+1}=R_{n+2}\ge 0$, so Proposition \ref{prop:alphaproperty}(ii) implies $\alpha_{n+1}\ge 1$. By Lemma \ref{lem:oddlatticeproperty2-1}, we also have $d((-1)^{(n+1)/2}a_{1,n+1})=1-R_{n+1}=1<d((-1)^{(n-1)/2}b_{1,n-1})$. Hence $d(-a_{1,n+1}b_{1,n-1})=1$ by the domination principle. So 
	\begin{align*}
		d[-a_{1,n+1}b_{1,n-1}]=\min\{d(-a_{1,n+1}b_{1,n-1}),\alpha_{n+1},\beta_{n-1}\}=1=1-R_{n+1},
	\end{align*}
 as required.
 
	(ii) Assume that $V:=[a_{1},\ldots,a_{n+1}]$ represents both $FC^n_1(c)$ and $FC^n_1(c\tilde{c}^{\#})$. Since $\det V=a_{1,n+1}=(-1)^{(n+1)/2}\tilde{c}$, by \cite[63:21]{omeara_quadratic_1963}, we have $
			FC_{1}^{n}(c)\perp [-c\tilde{c}]\cong FV \cong FC^n_1(c\tilde{c}^{\#})\perp [-c\tilde{c}\tilde{c}^{\#}]$
	and thus
	\begin{align*}
		\mathbb{H}^{(n-1)/2}\perp [c]\perp [-c\tilde{c}]\cong  \mathbb{H}^{(n-1)/2}\perp [c\tilde{c}^{\#}]\perp [-c\tilde{c}\tilde{c}^{\#}].
	\end{align*}  
 	 This implies $[c,-c\tilde{c}]\cong [c\tilde{c}^{\#}, -c\tilde{c}\tilde{c}^{\#}]$ by Witt's cancellation. Scaling by $c$, we get  $[1,-\tilde{c}]\cong [\tilde{c}^{\#}, -\tilde{c}\tilde{c}^{\#}]$. Hence $\tilde{c}^{\#}\rep [1,-\tilde{c}]$ and so $(\tilde{c}^{\#},\tilde{c})_{\mathfrak{p}}=1$, which contradicts \cite[Proposition 3.2]{HeHu2}.
\end{proof}
\begin{lem}\label{lem:classicalcon3odd}
	 Suppose that $ FM $ is $ n $-universal and $ M $ satisfies $ J_{1}^{\prime E}(n-1) $ in Lemma \ref{lem:classicalcon12even}, $ J_{2}^{E}(n-1) $ and $ J_{3}^{E}(n-1) $ in Theorem \ref{thm:classicaleven-nuniversaldyadic}. Then the following conditions are equivalent.
	\begin{enumerate}
		\item[\rm (i)] Theorem \ref{thm:beligeneral}(iii) holds for all n-ary $ \mathcal{O}_{F} $-lattices $ N $. 
		
		\item[\rm (ii)] Theorem \ref{thm:beligeneral}(iii) holds for $ N$ in the following list:
		\begin{align*}
			C_{1}^{n}(c),\;C_{1}^{n}(c\tilde{c}^{\#}),\;\text{with}\;c=(-1)^{(n+1)/2}a_{1,n+2}\quad\text{and}\quad \tilde{c}=(-1)^{(n+1)/2}a_{1,n+1}
		\end{align*}
		 (cf. \text{Definition \ref{defn:classicallattices}}), if either $ R_{n+1}=1 $ or $ R_{n+2}>1 $. 
		
		\item[\rm (iii)] $ M $ satisfies $ J_{2}^{O}(n) $ in Theorem \ref{thm:classicalodd-nuniversaldyadic}.
	\end{enumerate} 
\end{lem}
\begin{proof} 
	\textbf{(i)$ \Rightarrow $(ii)}: It is trivial.
	
	\textbf{(ii)$ \Rightarrow $(iii)}:  Assume $ \alpha_{n+2}>G_{n} $. If either $ R_{n+1}=1 $ or $ R_{n+2}>1 $, then, by Lemma \ref{lem:oddlatticeproperty2-2}, Theorem \ref{thm:beligeneral}(iii) fails at $ i=n+1 $ for either $ N=C_{1}^{n}(c) $ or $ N=C_{1}^{n}(c\tilde{c}^{\#}) $, which contradicts condition (ii). 		
	
	\textbf{(iii)$ \Rightarrow $(i)}: By $J_{1}^{\prime E}(n-1)$ and  $J_{2}^{E}(n-1)$, we have $R_{i}=0$ for $1\le i\le n$, $\alpha_{n}=1$ and $d[(-1)^{(n+1)/2}a_{1,n+1}]=1-R_{n+1}$. Hence Theorem \ref{thm:beligeneral}(iii) holds for $1\le i\le n$ by Corollary \ref{cor:B3}(iii).
	 	
	It remains to show Theorem \ref{thm:beligeneral}(iii) holds at the index $i=n+1$. First, $R_{n+2}\ge 0$ by Proposition \ref{prop:Ralphaproperty3}(i). Also, $R_{n+1}\in \{0,1\}$ by Lemma \ref{lem:0-1}. If $R_{n+1}=0$ and $R_{n+2}\in \{0,1\}$, then we are done by Lemma  \ref{lem:B3-n-n+1-odd}(ii). Hence we may let $R_{n+1}=1$ or $R_{n+2}>1$.
	 	
	Assume $d[-a_{1,n+1}b_{1,n-1}]+d[-a_{1,n+2}b_{1,n}]>2e+S_{n}-R_{n+2}$. By $ J_{2}^{O}(n) $, $ d[-a_{1,n+2}b_{1,n}]\le \alpha_{n+2}\le G_{n}  $. By  Lemma \ref{lem:An-odd-2}, $d[-a_{1,n+1}b_{1,n-1}]\le S_{n}-R_{n+1}+d[a_{1,n}b_{1,n}]$. Combining these with the assumption, we see that 
	 	\[
	 	S_{n}-R_{n+1}+d[a_{1,n}b_{1,n}]+G_n\ge  d[-a_{1,n+1}b_{1,n-1}]+d[-a_{1,n+2}b_{1,n}]>2e+S_{n}-R_{n+2}.
	 	\]It follows that
	 	\[
	  2e-R_{n+2}+R_{n+1}-G_n<d[a_{1,n}b_{1,n}]\le \alpha_{n}=1.
	 	\]
	 	But $2e-R_{n+2}+R_{n+1}-G_n\in \{0,1\}$ from \eqref{eq5.1new}. Hence we must have $ 2e-R_{n+2}+R_{n+1}-G_{n}=0 $ and $ d[a_{1,n}b_{1,n}]=1 $. The former equality implies that $ \ord(a_{n+1}a_{n+2})=R_{n+1}+R_{n+2} $ is odd; the latter equality implies that $ \ord(a_{1,n}b_{1,n}) $ is even. So $ \ord(a_{1,n+2}b_{1,n}) $ is odd and hence $ d[-a_{1,n+2}b_{1,n}]=0 $. Combining this with the assumption, we deduce that $R_{n+2}-R_{n+1}>2e-1$ by Lemma \ref{lem:An-odd-2}. This contradicts Lemma \ref{lem:simplifythm}.
\end{proof}
\begin{lem}\label{lem:classicalcon4odd}
	 Suppose that $ FM $ is n-universal and $ M $ satisfies $ J_{1}^{\prime E}(n-1) $ in Lemma \ref{lem:classicalcon12even}, $ J_{2}^{E}(n-1) $ and $ J_{3}^{E}(n-1) $ in Theorem \ref{thm:classicaleven-nuniversaldyadic} and $ J_{2}^{O}(n) $ in Theorem \ref{thm:classicalodd-nuniversaldyadic}. Then the following conditions are equivalent.
	\begin{enumerate}
		\item[\rm (i)] Theorem \ref{thm:beligeneral}(iv) holds for all n-ary $ \mathcal{O}_{F} $-lattices $ N $. 
		
		\item[\rm (ii)] Theorem \ref{thm:beligeneral}(iv) holds for the lattices $C_{1}^{n}(c)$ and $C_{2}^{n}(c)$, with $c=(-1)^{(n+1)/2}a_{1,n+2}$, if $ R_{n+3}-R_{n+2}>2e $ (cf. Definition \ref{defn:classicallattices}).
		
		\item[\rm (iii)] $ M $ satisfies $ J_{3}^{O}(n) $ in Theorem \ref{thm:classicalodd-nuniversaldyadic}.
	\end{enumerate}
\end{lem}
\begin{proof}
	\textbf{(i)$ \Rightarrow $(ii)}: It is trivial.
	
	\textbf{(ii)$ \Rightarrow $(iii)}: Assume $ R_{n+3}-R_{n+2}>2e $. By Lemma \ref{lem:0-1} with $j=n+1$, we have $R_{n+1}\in \{0,1\} $. By Proposition \ref{prop:Ralphaproperty3}(i), we also have $R_{n+2}\ge  0$. If either $ R_{n+1}=1 $ or $ R_{n+2}>1 $, then $ R_{n+3}-R_{n+2}\le 2e-1 $ by Lemma \ref{lem:simplifythm}, which contradicts the assumption. Thus $R_{n+1}\not=1$, i.e. $R_{n+1}=0$ and $ R_{n+2} \in \{0,1\}$.  
	
	Take $ N=C_{\nu}^{n}(c) $, with $ \nu\in \{1,2\} $ and $ c=(-1)^{(n+1)/2}a_{1,n+2} $. Since $R_{n}=0$, Proposition \ref{prop:Ralphaproperty3}(ii) implies that $ R_{i} $ is even for $ 1\le i\le n$. Since also $R_{n+1}=0$, $ \ord(c)=\ord( a_{1,n+2})\equiv R_{n+2}\pmod{2}$. Hence $ S_{n}=R_{n+2} $ by Lemma \ref{lem:S-invariant}(ii), so this combined with the assumption shows that the condition $ R_{n+3}>S_{n}+2e\ge R_{n+2}+2e$ is satisfied. Hence $ FN=[b_{1},\ldots,b_{n}]$ is represented by $V:=[a_{1},\ldots,a_{n+2}]  $ by Theorem \ref{thm:beligeneral}(iv). But $ \det V=a_{1,n+2}=(-1)^{(n+1)/2}c=-\det FN$, so $ V $ cannot represent both $ FC_{1}^{n}(c) $ and $ FC_{2}^{n}(c) $ by \cite[Lemma 3.13]{HeHu2}. A contradiction is derived and so $ R_{n+3}-R_{n+2}\le 2e $.
 
	\textbf{(iii)$ \Rightarrow $(i)}: It is clear from Corollary \ref{cor:B4}(iii).
\end{proof}
\begin{proof}[Proof of Proposition \ref{prop:classicalodd-nuniversaldyadic}]
	Recall from \eqref{equi:even} that
	\begin{align}\label{equi:odd}
		J_{1}^{\prime E}(n-1)\;\text{and}\; J_{2}^{E}(n-1)\Longleftrightarrow 	J_{1}^{E}(n-1)\;\text{and}\; J_{2}^{E}(n-1).
	\end{align}
	It is straightforward by Lemmas \ref{lem:classicalcon12odd}, \ref{lem:classicalcon3odd}, \ref{lem:classicalcon4odd} and Theorem \ref{thm:beligeneral}.
\end{proof}

\section{Proof of Theorems \ref{thm:classicalnuniversaldyadic} and \ref{thm:necessarycon}}\label{sec:proof-main}
In this section, we will prove Theorem \ref{thm:classicalnuniversaldyadic}, which provides a criterion not involving $ \alpha $-invariants. 
\begin{lem}\label{lem:J2E}
	Let $ n $ be an even integer and $n\ge 2$. Suppose that $M$ satisfies $ J_{1}^{E}(n) $ in Theorem \ref{thm:classicaleven-nuniversaldyadic}. Then Theorem \ref{thm:classicalnuniversaldyadic}(ii)(1) holds if and only if $M$ satisfies $ J_{2}^{E}(n) $ in Theorem \ref{thm:classicaleven-nuniversaldyadic}.
\end{lem}
\begin{proof}
	By $ J_{1}^{E}(n) $, we have $ R_{i}=0$ for $0\le i\le n+1 $. Note from Lemma \ref{lem:0-1} that $J_{2}^{E}(n)$ implies $R_{n+2}\in \{0,1\}$. 
	
	We may assume $R_{n+2}\in \{0,1\}$. If $R_{n+2}=1$, then $ \alpha_{n+1}=1 $ by Proposition \ref{prop:alphaproperty}(iii) and $d[(-1)^{(n+2)/2}a_{1,n+2}]=0=1-R_{n+2}$ by the odd parity of $\ord(a_{1,n+2})$.
	
	Suppose $R_{n+2}=0$. Then $R_{n+3}-R_{n+2}=R_{n+3}\ge 0$ by Proposition \ref{prop:Ralphaproperty3}(i), so  Proposition \ref{prop:alphaproperty}(ii) implies $\alpha_{n+2}\ge 1$. Since $\ord(a_{1,n+2})$ is even, $d((-1)^{(n+2)/2}a_{1,n+2})\ge 1$ and hence 
	\begin{align*}
	d[(-1)^{(n+2)/2}a_{1,n+2}]=\min\{d((-1)^{(n+2)/2}a_{1,n+2}), \alpha_{n+2}\}\ge 1.
	\end{align*}
	By Proposition \ref{prop:alphaproperty}(iii), we see that $\alpha_{n+2}=1$ if and only if $R_{n+3}\in \{0,1\}$.
	 So
	\begin{align*}
			d[(-1)^{(n+2)/2}a_{1,n+2}]=1  \iff d((-1)^{(n+2)/2}a_{1,n+2})=1\quad\text{or}\quad R_{n+3}\in \{0,1\}.
	\end{align*}
	 Thus under the condition (ii)(1)(a), $ J_{2}^{E}(n) $ holds if and only if $ \alpha_{n+1}=1 $.
	 
	 If $e=1$, since $ R_{n+2}-R_{n+1}=0$, Proposition \ref{prop:alphaproperty}(iii) implies $ \alpha_{n+1}=1 $.  
	
	If $ R_{n+3}=1 $, then $ \alpha_{n+2}=1$ by Proposition \ref{prop:alphaproperty}(iii) and so $ \alpha_{n+1}=1 $ by Proposition \ref{prop:Ralphaproperty3}(vi).  
	
	If $ d((-1)^{(n+2)/2}a_{1,n+2})=1 $, then, by \eqref{eq:alpha-defn} and the domination principle, we have
	\begin{align*}
		\alpha_{j-1}\le R_{j}-R_{j-1}+d(-a_{j-1}a_{j})=d(-a_{j-1}a_{j})=1
	\end{align*}
	 for some $ j\in [1,n+2]^{E} $. Hence $ \alpha_{n+1}=1 $ by Proposition \ref{prop:Ralphaproperty3}(vi). 
	
	Suppose the condition (ii)(1)(b) holds, i.e. $e>1$,  $ R_{n+2}=R_{n+3}=0 $ and $d((-1)^{(n+2)/2}a_{1,n+2})>1$. Recall \eqref{T} and write $ T_{j}=T_{j}^{(n+1)} $ for $ 0\le j\le m-1 $ for short. Then $ \alpha_{n+1}=\min\{T_{0},\ldots,T_{m-1}\} $. Note that
	\begin{align*}
		T_{0}=\dfrac{R_{n+2}-R_{n+1}}{2}+e=e>1.
	\end{align*}
	So $ \alpha_{n+1}=1 $ if and only if $ 1\in \{T_{1},\ldots,T_{m-1}\} $. This is equivalent to, either 
	\begin{align*}
				1=T_{j}=R_{n+2}-R_{j}+d(-a_{j}a_{j+1})&=d(-a_{j}a_{j+1})=R_{j+1}+d(-a_{j}a_{j+1})
	\end{align*}
	for some $j $ with $1\le j\le n+1 $, or 
	\begin{align*}
		1=T_{j}=R_{j+1}-R_{n+1}+d(-a_{j}a_{j+1})=R_{j+1}+d(-a_{j}a_{j+1})
	\end{align*}
 for some $j $ with $ n+2\le j\le m-1 $. 
\end{proof}
\begin{lem}\label{lem:J2O}
	Let $ n  $ be an odd integer and $n\ge 3$. Suppose that $M$ satisfies $ J_{1}^{O}(n) $ in Theorem \ref{thm:classicalodd-nuniversaldyadic}. Then Theorem \ref{thm:classicalnuniversaldyadic}(iii)(2) holds if and only if $M$ satisfies $ J_{2}^{O}(n) $ in Theorem \ref{thm:classicalodd-nuniversaldyadic}.
\end{lem}
\begin{proof}
	Write $ T_{j}=T_{j}^{(n+2)} $ for $ 0\le j\le m-1 $ for short (cf. \eqref{T}). Then $ \alpha_{n+2} =\min\{T_{0},\ldots,T_{m-1}\}$. We may suppose either $ R_{n+1}=R_{n+2}=1 $ or $ R_{n+2}>1 $. By $J_{1}^{O}(n)$, we have $R_{i}=0$ for $1\le i\le n$ and $\alpha_{n}=1$. By Proposition \ref{prop:Ralphaproperty3}(iv), we have $R_{n+1}\ge 0$, so Proposition \ref{prop:alphaproperty}(iii) implies $R_{n+1}\in \{0,1\}$. Hence
	\begin{align}
               &-R_{i}+d(-a_{i}a_{i+1})=d(-a_{i}a_{i+1})\ge 1\ge 1-R_{n+1}\quad\text{for $ 1\le i\le n-1$}, \label{j=1-n-1}	\\ 
	 &-R_{n}+d(-a_{n}a_{n+1})\ge \alpha_{n}-R_{n+1}=1-R_{n+1}\quad(\text{by \eqref{eq:alpha-defn}}).\label{j=n}
	 	\end{align}
	 We claim that $ T_{j}+G_{n}\ge 2T_{0} $ for $1\le j\le n+1$. By \eqref{eq5.1new}, $t:=2e-R_{n+2}+R_{n+1}-G_n\in \{0,1\}$. We have
	\begin{align*}
		T_{j}+G_{n}&=(R_{n+3}-R_{j}+d(-a_{j}a_{j+1}))+(2e-R_{n+2}+R_{n+1}-t)\\
		&=(R_{n+3}-R_{n+2}+2e)+(R_{n+1}-R_{j}+d(-a_{j}a_{j+1})-t)\\
		&=2T_{0}+R_{n+1}-R_{j}+d(-a_{j}a_{j+1})-t.
	\end{align*}
	It is sufficient to show that $ R_{n+1}-R_{j}+d(-a_{j}a_{j+1})-t\ge 0 $. For $1\le j\le n$, by \eqref{j=1-n-1} and \eqref{j=n}, we have $  R_{n+1}-R_{j}+d(-a_{j}a_{j+1})-t\ge R_{n+1}+(1-R_{n+1})-t\ge 0 $. For $j=n+1$, by \eqref{eq5.1new}, we have
	\begin{align*}
		d(-a_{n+1}a_{n+2})
		\begin{cases}
			\ge 1=t   &\text{if $R_{n+2}-R_{n+1}$ is even},  \\
			=0=t       &\text{if $R_{n+2}-R_{n+1}$ is odd}.
		\end{cases}
	\end{align*}
	 In both cases we have $ R_{n+1}-R_{n+1}+d(-a_{n+1}a_{n+2})-t=d(-a_{n+1}a_{n+2})-t\ge 0 $. Thus the claim is proved.
	
	Note that $ \alpha_{n+2}=\min\{T_{0},\ldots,T_{m-1}\}\le G_{n} $ if and only if $ T_{k}\le G_{n} $ for some $k\in \{0,\ldots,m-1\} $. But for $ 1\le j\le n+1 $, if $ T_{j}\le G_{n}$, then, by the claim,  $ 2T_{0}\le T_{j}+G_{n}\le 2G_{n}$, equivalently, $ T_{0}\le G_{n} $. Hence $ \alpha_{n+2}\le G_{n} $ if and only if $ T_{k}\le G_{n} $ for some $ k\in \{0,n+2,\ldots,m-1\} $.
	
	Now, one can check that
	\begin{align*}
		T_{0}\le G_{n}\quad&\Longleftrightarrow\quad\dfrac{R_{n+3}-R_{n+2}}{2}+e\le 2e-R_{n+2}+R_{n+1}-t\\
		&\Longleftrightarrow\quad R_{n+3}+R_{n+2}-2R_{n+1}\le 2e-2t
	\end{align*}
	and 
	\begin{align*}
		T_{j}\le G_{n}\quad& \Longleftrightarrow\quad R_{j+1}-R_{n+2}+d(-a_{j}a_{j+1})\le 2e-R_{n+2}+R_{n+1}-t\\
		& \Longleftrightarrow\quad d(-a_{j}a_{j+1})\le 2e+R_{n+1}-R_{j+1}-t
	\end{align*}
 for $n+2\le j\le m-1$. Recall from \eqref{eq5.1new} that $ t=1 $ or $ 0 $ accordingly as $ R_{n+2}-R_{n+1} $ is even or odd, so these inequalities agree with those in Theorem \ref{thm:classicalnuniversaldyadic}(iii)(2).	
\end{proof}
\begin{proof}[Proof of Theorem \ref{thm:classicalnuniversaldyadic}]
 Note that the condition $ J_{2}^{E}(n) $ implies $ m\ge 5 $ if $ n=2 $. Hence, by \cite[Theorem 2.3]{hhx_indefinite_2021}, $ FM $ is $ n $-universal if and only if $ m\ge n+3 $.
	
If $ n $ is even, then we have the following equivalence:
	\begin{align*}
		&J_{1}^{E}(n)\iff \text{(i)}\;\text{and}\;R_{n+1}=0;\;\; J_{3}^{E}(n)\iff  \text{(ii)(2)}; \\ &J_{2}^{E}(n)\iff  \text{(ii)(1)} \;\;\text{(by Lemma \ref{lem:J2E})}.
	\end{align*}
Hence we are done by Theorem \ref{thm:classicaleven-nuniversaldyadic}.
	If $ n $ is odd, then we have the following equivalence:
	\begin{align*}
		&J_{1}^{O}(n)\iff \text{(i) and (iii)(1)} \;\;\text{(by Lemma \ref{lem:J2E})};\\
		&J_{2}^{O}(n)\iff  \text{(iii)(2)} \;\;\text{(by Lemma \ref{lem:J2O})};\;\;J_{3}^{O}(n)\iff \text{(iii)(3)}.
	\end{align*}
Hence we are done by Theorem \ref{thm:classicalodd-nuniversaldyadic}.
\end{proof}
\begin{proof}[Proof of Theorem \ref{thm:necessarycon}]
	Assume $ e:=e_{\mathfrak{p}}>1 $. Let $ M:=L_{\mathfrak{p}}\cong \prec a_{1},\ldots,a_{m}\succ $ relative to some good BONG, $R_{i}=R_{i}(M)$ and $\alpha_{i}=\alpha_{i}(M)$. Suppose $n$ to be odd. Since $M$ is classic $n$-universal, we have $R_{i}=0$ for $1\le i\le n$ and $\{R_{n+1},R_{n+2}\}\subseteq \{0,1\}$ by  \cite[Theorem 2.1]{beli_universal_2020} and Theorem \ref{thm:classicalnuniversaldyadic}(i) and (iii)(1). The hypothesis $d(a_{n}a_{n+1})>0$ implies that $R_{n+1}$ is even, so $R_{n+1}=0$. Similarly, we have $ R_{n+2}=0$. Also, Proposition \ref{prop:Ralphaproperty3}(iv) implies that $R_{i}\ge 0$ for $1\le i\le m$.
	
	We have $ d(a_{i}a_{i+1})>1 $ for $1\le i\le m-1$ from the hypothesis. Since $ d(-1)\ge e>1$ by \cite[Lemma 2.1]{HeHu3}, we have $ d(-a_{i}a_{i+1})>1 $ by the domination principle. 
	
	If $ n=1 $, then $T_{0}^{(1)}=e>1$ and $T_{k}^{(1)}=R_{k+1}-R_{1}+d(-a_{k}a_{k+1})\ge d(-a_{k}a_{k+1})>1$ for $1\le k\le m-1$ (cf. \eqref{T}). Hence $\alpha_{1}=\min\{T_{0}^{(1)},\ldots,T_{m-1}^{(1)}\}>1$, which contradicts \cite[Theorem 2.1]{beli_universal_2020}. If $n\ge 3$, then  $d((-1)^{(n+1)/2}a_{1,n+1})>1$ by the domination principle. Combining this with $e>1$ and $R_{n+1}=R_{n+2}=0$, we conclude that $d(-a_{j}a_{j+1})=1-R_{j+1}\le 1$ for some $1\le j\le n$ by Theorem \ref{thm:classicalnuniversaldyadic}(iii)(1)(b), a contradiction; a similar argument can be applied for even $n\ge 2$. With above discussion, we deduce $e=1$.
\end{proof}
\begin{cor}\label{cor:diagonalizable}
	If $M\cong \prec a_{1},\ldots, a_{n+3}\succ$ is $n$-universal, then $M$ is diagonalizable and $M\cong \langle a_{1},\ldots, a_{n+3}\rangle$.
\end{cor}
\begin{proof}
	Without loss of generality, assume that $n\ge 2$ is even. Since $M\cong \prec a_{1},\ldots, a_{n+3}\succ$ is $n$-universal, by Theorem \ref{thm:classicalnuniversaldyadic}, we have $R_{i}=0$ for $1\le i\le n+1$ and $R_{n+2}\in \{0,1\}$. By Proposition \ref{prop:Ralphaproperty3}(i), we also have $R_{n+3}\ge 0$, moreover, $R_{n+3}\ge 1$ by Proposition \ref{prop:Rproperty}(iv) provided that $R_{n+2}=1$. Hence the sequence $R_{i}$ ($1\le i\le n+3$) is non-decreasing. So $M$ is diagonalizable and $M\cong \langle a_{1},\ldots, a_{n+3}\rangle$ by \cite[Corollaries 3.4(ii) and 4.4(i)]{beli_integral_2003}.
\end{proof}

\section{Proof of Theorem \ref{thm:classicalnuniversaldyadic15theorem}} \label{sec:classicalcstheorem-mini}
The minimal set for tesing $ n $-universal integral lattices was established in \cite[Proposition 3.2]{hhx_indefinite_2021} and \cite[Theorem 1.2]{HeHu2} by finding all rank $ n $ maximal lattices (in the sense of \cite[\S 82H]{omeara_quadratic_1963}). However, such approach cannot be applied to classic integral cases directly because a maximal lattice may not be classic integral (see \cite[82:21]{omeara_quadratic_1963} for example). Instead, we determine a minimal testing set by introducing some auxiliary lattices.
\begin{lem}\label{lem:rep-even-odd}
	Let $N$ be an $\mathcal{O}_{F}$-lattice of even rank $n\ge 2$. Let $ \omega $ be the unit as in Definition \ref{defn:classicallattices}.   
	\begin{enumerate}
		\item[\rm (i)] We have either $FN\rep FC_{1}^{n+1}(\omega)$ or $FN\rep FC_{2}^{n+1}(\omega)$. 
		
		\item[\rm (ii)] We have either $N\rep C_{1}^{n+1}(\omega)$ or $N\rep C_{2}^{n+1}(\omega)$.
	\end{enumerate}

	 Thus if an $\mathcal{O}_{F}$-lattice represents both of $C_{1}^{n+1}(\omega)$ and $C_{2}^{n+1}(\omega)$, then it is $n$-universal.
\end{lem}
\begin{proof}
	(i) Let $\det N=c$. Then $FN\perp [(-1)^{n/2}c\omega]\cong W_{1}^{n+1}(\omega)$ or $W_{2}^{n+1}(\omega)$ by \cite[Proposition 3.5(ii)]{HeHu2}. Since $\det N\det W_{j}^{n+1}(\omega)=(-1)^{n/2}c\omega$ ($j=1,2$), we have either $FN\rep W_{1}^{n+1}(\omega)=FC_{1}^{n+1}(\omega)$ or $FN\rep W_{2}^{n+1}(\omega)=FC_{2}^{n+1}(\omega)$ by \cite[63:21]{omeara_quadratic_1963}.
	
	(ii) Take $M=C_{j}^{n+1}(\omega)$, $j=1,2$. By Lemma \ref{lem:S-invariant}(iii), we have $R_{i}=0$ for $1\le i\le n$, $R_{n+1}=1-d(\omega)$ and $\alpha_{n}=1$. Then, by the domination principle, we see that
	\begin{align*}
		d((-1)^{(n+1)/2}a_{1,n+1})=d(-a_{n}a_{n+1})=d(\omega)=1=1-R_{n+1}.
	\end{align*}
	Hence $M$ satisfies the conditions in Lemma \ref{lem:m=n+1representation}, so either $N\rep C_{1}^{n+1}(\omega)$ or $N\rep C_{2}^{n+1}(\omega)$ by (i) and Lemma \ref{lem:m=n+1representation}. (ii) is proved.
	
	Clearly, if an $\mathcal{O}_{F}$-lattice $L$ represents both $C_{1}^{n+1}(\omega)$ and $C_{2}^{n+1}(\omega)$, then $L$ represents all lattices with rank $n$ by (ii), i.e. $L$ is $n$-universal.
\end{proof}
\begin{cor}\label{cor:J2En-1}
	Let $ n $ be an odd integer and $n\ge 3$. If $ M $ represents $ C_{1}^{n}(\omega) $ and $ C_{2}^{n}(\omega) $, then it satisfies $J_{1}^{\prime E}(n-1)$ in Lemma \ref{lem:classicalcon12even}, $ J_{2}^{E}(n-1) $ and $ J_{3}^{E}(n-1) $ in Theorem \ref{thm:classicaleven-nuniversaldyadic}.
\end{cor}
\begin{proof}
    If $ M $ represents $ C_{1}^{n}(\omega) $ and $ C_{2}^{n}(\omega) $, then it is $(n-1)$-universal by Lemma \ref{lem:rep-even-odd}. Hence $ M $ satisfies $J_{1}^{\prime E}(n-1)$, $ J_{2}^{E}(n-1) $ and $ J_{3}^{E}(n-1) $ by equivalence \eqref{equi:even} and Theorem \ref{thm:classicaleven-nuniversaldyadic}. 
\end{proof}
\begin{lem}\label{lem:FM n-universal}
	Let $ n $ be an integer and $ n\ge 2 $. If an $ \mathcal{O}_{F} $-lattice $M$ represents all lattices in $ \mathcal{C}_{e}^{n}  $, then $FM$ is $n$-universal.
\end{lem}
\begin{proof}
	Suppose that $ M $ represents all lattices in $ \mathcal{C}_{e}^{n} $. 
	
	For even $n\ge 2$, since $M$ represents $H_{e}^{n}(1)$ and $C_{1}^{n}(\omega)$, it satisfies $J_{1}^{\prime E}(n)$ and $J_{2}^{\prime E}(n)$ by Lemma \ref{lem:classicalcon12even}. Since $M$ also represents $C_{2}^{n}(\omega) $ and $ H_{e}^{n}(\Delta) $ (if $ e=1 $), it satisfies $J_{2}^{E}(n)$ by Lemma \ref{lem:classicalcon3even}.
Assume $m=n+2$. Then
	\begin{align*}
		d((-1)^{(n+2)/2}a_{1,n+2})=d[(-1)^{(n+2)/2}a_{1,n+2})]\in \{0,1\}
	\end{align*}
	by Lemma \ref{lem:0-1}. Note that $ M $ represents $
		C_{1}^{n}(c)$ and $C_{2}^{n}(c)$, with $c\in F^{\times}/F^{\times 2}$ and $ d(c)\in \{0,1\}$. In particular,
	$M$ represents $C_{1}^{n}(c)$ and $C_{2}^{n}(c)$ with $c=(-1)^{(n+2)/2}a_{1,n+2}$. Thus $FM$ represents both $FC_{1}^{n}(c)$ and $FC_{2}^{n}(c)$. But this contradicts \cite[Lemma 3.13]{HeHu2} because of $\det FM=a_{1,n+2}=-\det FC_{j}^{n}(c)$ ($j=1,2$). Hence $ m\ge n+3 $ and so $ FM $ is $ n $-universal by \cite[Theorem 2.3]{hhx_indefinite_2021}.
	
	For odd $n\ge 3$, a direct computation shows that the spaces spanned by lattices in $ \mathcal{C}_{e}^{n} $ exhaust all possible $ n $-ary quadratic spaces  by \cite[Proposition 3.5(ii)]{HeHu2} and thus $ FM $ is $ n $-universal.
\end{proof} 
\begin{lem}\label{lem:classicalgivenM-15theorem}
	Let $ n $ be an integer and $ n\ge 2 $. Then an $ \mathcal{O}_{F} $-lattice $M$ is $n$-universal if and only if it represents all lattices in $ \mathcal{C}_{e}^{n}  $. 
\end{lem}
\begin{proof}
     Necessity is clear. For sufficiency, by Lemma \ref{lem:FM n-universal}, $FM$ is $n$-universal.
     
     For even $n\ge 2$, a direct computation shows that the lattices $H_{e}^{n}(1)$, $H_{e}^{n}(\Delta)$ (if $e=1$), $C_{1}^{n}(\omega)$,  $C_{2}^{n}(\omega)$, and $C_{1}^{n}(c)$ and $C_{2}^{n}(c)$, with $c\in F^{\times}/F^{\times 2}$ and $d(c)\in \{0,1\}$,  
	are contained in $ \mathcal{C}_{e}^{n} $. Hence the testing $ n $-universality of $ \mathcal{C}_{e}^{n} $ follows by Lemmas \ref{lem:classicalcon12even}, \ref{lem:classicalcon3even}, \ref{lem:classicalcon4even}, equivalence \eqref{equi:even} and Theorem \ref{thm:classicaleven-nuniversaldyadic}.
	
    For odd $n\ge 3$, since $M$ represents $ C_{1}^{n}(\omega) $ and $ C_{2}^{n}(\omega) $, it satisfies $J_{1}^{\prime E}(n-1)$, $J_{2}^{E}(n-1)$ and $J_{3}^{E}(n-1)$ by Corollary \ref{cor:J2En-1}. Also, for any $\varepsilon\in \mathcal{O}_{F}^{\times}$, the lattices $ C_{1}^{n}(\varepsilon) $, $ C_{2}^{n}(\varepsilon) $, $ C_{1}^{n}(\varepsilon\pi) $ and $ C_{2}^{n}(\varepsilon\pi) $
	are contained in $ \mathcal{C}_{e}^{n} $. Combining Lemmas \ref{lem:classicalcon12odd}, \ref{lem:classicalcon3odd}, \ref{lem:classicalcon4odd}, equivalence \eqref{equi:odd} and Proposition \ref{prop:classicalodd-nuniversaldyadic}, we show the testing $ n $-universality in this case.
\end{proof}
For even $n\ge 4$ and $c\in F^{\times}/F^{\times 2}$, we define the rank $n$ lattice
\begin{align*}
	P_{k}^{n}(c):=\mathbf{H}_{0}^{(n-4)/2}\perp
	\begin{cases}
	 \prec 1,-c,-1,c\succ &\text{if $k=1$}, \\
		 \prec 1,-c,-c^{\#},cc^{\#}\succ &\text{if $k=2$}.  
	\end{cases}
\end{align*}  
\begin{lem}\label{lem:P-spacerep}
	Suppose that $n\ge 2$ is even.  Let $M=P_{k}^{n+2}(a)$, with $k\in \{1,2\}$ and $a\in F^{\times}/F^{\times 2}$, and $N=C_{j}^{n}(c)$, with $j\in \{1,2\}$, $c\in F^{\times}/F^{\times 2}$ and $d(c)\in \{0,1\}$. Then
	\begin{enumerate}
	   \item[\rm (i)] $FN\rep FM$.
		
		\item[\rm (ii)] $FH_{e}^{n}(1)\nrep FP_{2}^{n+2}(\Delta)$.
		
		 \item[\rm (iii)] $FH_{1}^{n}(1)\rep FP_{1}^{n+2}(\omega)$ and  $FH_{1}^{n}(\Delta)\rep FP_{2}^{n+2}(\omega)$.
	\end{enumerate}
\end{lem}
\begin{proof}
	For (i), we have $\det FM=(-1)^{(n-2)/2}$ and $\det FN=(-1)^{n/2}c$. Since $d(c)\in \{0,1\}$, $-\det FN\det FM=c\not=1$. So $ FN\rep  FM $ by \cite[63:21]{omeara_quadratic_1963}. 
	
	For (ii) and (iii), we may assume $n=2$. Then $FH_{e}^{2}(1)\cong \mathbb{H}$ for $e\ge 1$ and $FP_{2}^{4}(\Delta)\cong [1,-\Delta,-\pi,\Delta\pi]\cong [1,-\Delta,\pi,-\Delta\pi]$. Thus (ii) follows by \cite[63:17]{omeara_quadratic_1963}.
	
	 By definition, $FP_{1}^{4}(\omega)\cong [1,-\omega,-1,\omega]$. By Proposition \ref{prop:isotropic-ternary}, $[1,-\omega,-1]$ is isotropic, so  $FP_{1}^{4}(\omega)$ represents $\mathbb{H}$. By Proposition \ref{prop:isotropic-ternary} again, $ [-\omega,-\omega^{\#},\omega\omega^{\#}]$ is anisotropic, so it is isometric to $ [-\Delta,\pi,-\Delta\pi]$ by \cite[Proposition 3.5(ii)]{HeHu2}. Hence $FP_{2}^{4}(\omega)\cong [1,-\omega,-\omega^{\#},\omega\omega^{\#}]\cong [1,-\Delta,\pi,-\Delta\pi] $ by \cite[63:17]{omeara_quadratic_1963}, thereby representing
	 $FH_{1}^{2}(\Delta)\cong [\pi,-\Delta\pi]$. 
\end{proof}
\begin{lem}\label{lem:P-invariants}
	Let $n$ be an even integer and $n\ge 2$.
	\begin{enumerate}
		\item[\rm (i)] If $M=P_{2}^{n+2}(\Delta)  $, then $R_{i}=0$ and $\alpha_{i}=1$ for $1\le i\le n$, and $R_{n+1}=R_{n+2}=1$.		
		
		\item[\rm (ii)] If $M=P_{k}^{n+2}(\omega)$, $k=1,2$, then $R_{i}=0$ for $1\le i\le n+2$ and $\alpha_{i}=1$ for $1\le i\le n+1$.
	\end{enumerate}
\end{lem}
\begin{proof}
	(i) By Lemma \ref{lem:BONGformaximallattice}, we have $R_{i}=0$ for $1\le i\le n$ and $R_{n+1}=R_{n+2}=1$. Since $R_{n+1}-R_{n}=1$, Proposition \ref{prop:alphaproperty}(iii) implies $\alpha_{n}=1$. So $\alpha_{i}=1$ for $1\le i\le n-1$ by Proposition \ref{prop:Ralphaproperty3}(vi).
	
	(ii) By Lemma \ref{lem:BONGformaximallattice}, we have $R_{i}=0$ for $1\le i\le n+2$. Since $R_{n}-R_{n-1}+d(-a_{n-1}a_{n})=d(\omega)=1$, Proposition \ref{prop:alphaproperty}(vi) implies $\alpha_{n-1}=1$. Hence $\alpha_{i}=1$ for $1\le i\le n+1$ by Proposition \ref{prop:Ralphaproperty3}(vi).
\end{proof}
\begin{lem}\label{lem:PDelta}
Suppose that $n\ge 2$ is even. Let $M=P_{2}^{n+2}(\Delta)$ and $N=C_{j}^{n}(c)$, with $j\in \{1,2\}$, $c\in F^{\times}/F^{\times 2}$ and $d(c)\in \{0,1\}$.
	\begin{enumerate}
		\item[\rm (i)] Theorem \ref{thm:beligeneral}(ii) holds at $i=n-2,n-1,n$ for $M$ and $N$.
		
		\item[\rm (ii)] Theorem \ref{thm:beligeneral}(iii) holds at $i=n-1,n,n+1$ for $M$ and $N$.
		
		\item[\rm (iii)] Theorem \ref{thm:beligeneral}(i)-(iv) holds for $M$ and $N$.
	\end{enumerate}
\end{lem}
\begin{proof}
	 By Lemma \ref{lem:P-invariants}(i), we have $R_{i}=0$ and $\alpha_{i}=1$ for $1\le i\le n$, and $R_{n+1}=R_{n+2}=1$. By Lemma \ref{lem:S-invariant}(iii), we also have $S_{i}=0$ for $1\le i\le n-1$, $S_{n}=1-d(c)$ and $\beta_{i}=1$ for $1\le i\le n-1$.
	
	(i) Let $i\in \{n-2,n-1\}$. Since $\ord(a_{1,i}b_{1,i}) $ is even, $d(a_{1,i}b_{1,i})\ge 1$. Since also $\alpha_{i}=\alpha_{i+1}=\beta_{i}=1$, we have
	\begin{align*}
		d[a_{1,i}b_{1,i}]=\min\{d(a_{1,i}b_{1,i}),\alpha_{i},\beta_{i}\}=1=\alpha_{i+1}.
	\end{align*}
	Combining with $R_{i+1}=S_{i}=0$, we deduce that
	\begin{align}\label{Ai}
		A_{i}\le R_{i+1}-S_{i}+d[-a_{1,i+1}b_{1,i-1}]=d[-a_{1,i+1}b_{1,i-1}]\le \alpha_{i+1}=d[a_{1,i}b_{1,i}].
	\end{align}
 	Note that $d(a_{1,n}b_{1,n})\ge 1$ or $=0$, according as $ S_{n}=0$ or $1$. Also, $\alpha_{n}=1$. Hence
   	\begin{align*}
   		d[a_{1,n}b_{1,n}]=\min\{d(a_{1,n}b_{1,n}),\alpha_{n}\}=
   		\begin{cases}
   			1  &\text{if $S_{n}=0$},  \\
   			0  &\text{if $S_{n}=1$}.
   		\end{cases}
   	\end{align*}
   	Therefore, $d[a_{1,n}b_{1,n}]=1-S_{n}$. Since $\ord(a_{1,n+1}b_{1,n-1})$ is odd, $d[-a_{1,n+1}b_{1,n-1}]=0$. Also, $R_{n+1}=1$. Hence
	\begin{align}\label{An}
		A_{n}&\le R_{n+1}-S_{n}+d[-a_{1,n+1}b_{1,n-1}]=1-S_{n}+0=1-S_{n}=d[a_{1,n}b_{1,n}].
	\end{align} 
	Combining \eqref{Ai} with \eqref{An}, we see that Theorem \ref{thm:beligeneral}(ii) holds for the indices $n-2,n-1$ and $n$.
	
	(ii) Since $S_{n-1}=0$ and $S_{n}=1-d(c)\ge 0$ and $R_{n+1}=R_{n+2}=1$, we have
	\begin{align}\label{2e-1}
		2e-1=2e+S_{n-1}-R_{n+1}\le 2e+S_{n}-R_{n+2}.
	\end{align}
	Note from (i) that $d[-a_{1,n+1}b_{1,n-1}]=0$ and $d[-a_{1,n}b_{1,n-2}]\le \alpha_{n}=1$. Hence 
	\begin{align}\label{B3-n}
		 d[-a_{1,n}b_{1,n-2}]+d[-a_{1,n+1}b_{1,n-1}]\le  1+0\le 2e-1=2e+S_{n-1}-R_{n+1}
	\end{align}
	by the equality in \eqref{2e-1}. Since $d(-a_{i-1}a_{i})\ge 2e$ for $i\in [1,n+2]^{E}$ and $d(-b_{i-1}b_{i})\ge 2e$ for $i\in [1,n-2]^{E}$, by the domination principle, we have 
	\begin{align*}
		d[-a_{1,n+2}b_{1,n}] \le 	d(-a_{1,n+2}b_{1,n})=d(-b_{n-1}b_{n})=d(c)\le 1.
	\end{align*}
 	Hence
	\begin{align}\label{B3-n+1}
		d[-a_{1,n+1}b_{1,n-1}]+d[-a_{1,n+2}b_{1,n}] \le 0+ 1\le 2e-1\le 2e+S_{n}-R_{n+2}
	\end{align}
	by the inequality in \eqref{2e-1}. Combining $R_{n}=0=S_{n-2}$, \eqref{B3-n} and \eqref{B3-n+1}, we see that Theorem \ref{thm:beligeneral}(iii) holds for the indices $n-1,n$ and $n+1$.
	
	(iii) Recall the invariants of $M$ and $N$. Theorem \ref{thm:beligeneral}(i) holds for $1\le i\le n$ by Corollary \ref{cor:B1}(i). Theorem \ref{thm:beligeneral}(ii) holds for $1\le i\le n-3$ by Corollary \ref{cor:B2}(i) and holds for $n-2\le i\le n$ by (i). Theorem \ref{thm:beligeneral}(iii) holds for $2\le i\le n-2$ by Corollary \ref{cor:B3}(i) and holds for $n-1\le i\le n+1$ by (ii). Theorem \ref{thm:beligeneral}(iv) holds for $2\le i\le n$ by Corollary \ref{cor:B4}(ii). 
\end{proof}
\begin{lem}\label{lem:Pomega-1}
 Suppose that $n\ge 2$ is even and $e=1$. Let $M=P_{k}^{n+2}(\omega)$ with $k\in \{1,2\}$.
	\begin{enumerate}
		 \item[\rm (i)] Theorem \ref{thm:beligeneral}(i)(ii) and (iv) holds for $M$ and any rank $n$ lattice $N$; 
		 
		 \item[\rm (ii)] Let $N=C_{j}^{n}(c)$, with $j\in\{1,2\}$, $c\in F^{\times}/F^{\times 2}$ and $d(c)\in \{0,1\}$. Then Theorem \ref{thm:beligeneral}(iii) holds for $M $ and $N$. 
		 
		 \item[] Thus Theorem \ref{thm:beligeneral}(i)-(iv) holds for $M$ and $N$.
		 
		 \item[\rm (iii)] Let $N=H_{1}^{n}(\mu)$ with $\mu\in \{1,\Delta\}$. Then Theorem \ref{thm:beligeneral}(iii) holds for $2\le i\le n$.
	\end{enumerate}
\end{lem}
\begin{proof}
	By Lemma \ref{lem:P-invariants}(ii), we have $R_{i}=0$ for $1\le i\le n+2$ and $\alpha_{i}=1$ for $1\le i\le n+1$. 
	
	(i) The first assertion follows by Corollaries \ref{cor:B1}(i) \ref{cor:B2}(ii) and \ref{cor:B4}(ii). 

	(ii) By Lemma \ref{lem:S-invariant}(iii), we have $S_{i}=0$ for $1\le i\le n-1$ and $S_{n}\in \{0,1\}$. Since $R_{i+1}=0\le S_{i-1} $ for $2\le i\le n+1$, Theorem \ref{thm:beligeneral}(iii) holds for $2\le i\le n+1 $.

	(iii) By Lemma \ref{lem:S-invariant}(i), we have $S_{i-1}=-S_{i}=e=1$ for $i\in [1,n]^{E}$.

	For $ i\in [2,n]^{O} $, we have $S_{i-1}-S_{i-2}=-2=-2e$, so $\beta_{i-2}=0$ by Proposition \ref{prop:alphaproperty}(ii). Since $R_{i+1}=0$ and $\alpha_{i+1}=1$, we have 
	\begin{align*}
		d[-a_{1,i}b_{1,i-2}]+d[-a_{1,i+1}b_{1,i-1} ]\le 	\beta_{i-2} +\alpha_{i+1}=0+1\le 2e-1=2e+S_{i-1}-R_{i+1}.
	\end{align*}
	For $ i\in [2,n]^{E} $, we have $ R_{i+1}=0<1=S_{i-1}$. So Theorem \ref{thm:beligeneral}(iii) holds for $ 2\le i\le n $. 
\end{proof}
\begin{lem}\label{lem:Pomega-2}
	 Suppose that $n\ge 2$ is even and $e=1$. Let $M_{k}=P_{k}^{n+2}(\omega)$ with $k\in \{1,2\}$, $N_{1}=H_{1}^{n}(1)$ and $N_{2}=H_{1}^{n}(\Delta)$.
	\begin{enumerate}
		 \item[\rm (i)] We have $R_{n+2}>S_n$ and $d[-a_{1,n+1}b_{1,n-1}]+d[-a_{1,n+2}b_{1,n}]>2e+S_n-R_{n+2}$ for $M=M_{k}$ and $N=N_{j}$, where $k,j\in \{1,2\}$.
		
		 \item[\rm (ii)] The associated $(n+1)$-ary space of $FP_{k}^{n+2}(\omega)$, $[a_1,\ldots, a_{n+1}]_{k}$, represents $FN_{j}=[b_1,\ldots, b_n]_{j}$ or not, according as $k=j$ or $k\not=j$.
		
		\item[] Thus for $M_{k}$ and $N_{j}$, Theorem \ref{thm:beligeneral}(iii) holds at $ i=n+1 $ if $k=j$ and fails at $i=n+1$ if $k\not=j$.
		
		 \item[\rm (iii)] If $j=k$, then $N_{j}\rep M_{k}$; if $j\not=k$, then $N_{j}\nrep M_{k}$.  
	\end{enumerate} 
\end{lem}
\begin{proof}
	(i) By Lemmas \ref{lem:P-invariants}(ii) and \ref{lem:S-invariant}(i), we have
	 $ R_{n+2}=0>-1=S_{n} $.  A direct computation shows that $d[-a_{1,n+2}b_{1,n} ]=d(-a_{1,n+2}b_{1,n} )\ge 2e$. Hence
	\begin{align*}
		d[-a_{1,n+1}b_{1,n-1}]+d[-a_{1,n+2}b_{1,n} ]\ge d[-a_{1,n+2}b_{1,n}]\ge 2e>2e-1=2e+S_{n}-R_{n+2}.
	\end{align*}

	(ii) Write $V_{k}:=[a_{1},\ldots,a_{n+1}]_{k}$. By definition and Proposition \ref{prop:isotropic-ternary}, $V_{k}\cong \mathbb{H}^{(n-2)/2}\perp U_{k}$ for some ternary space $U_{k}$, where $U_{1}\cong [1,-\omega, -1]$ is isotropic and $U_{2}\cong [1,-\omega,-\omega^{\#}]$ is anisotropic. On the other hand, $FN_{1}\cong \mathbb{H}^{n/2}$ and $FN_{2}\cong \mathbb{H}^{(n-2)/2}\perp [\pi,-\Delta\pi]$. Hence Lemma \ref{lem:spacerep-criterion}(ii) shows that $FN_{1}\rep V_{1}$ and $FN_{2}\nrep V_{1}$; Lemma \ref{lem:spacerep-criterion}(iii) shows that $FN_{2}\rep V_{2}$ and $FN_{1}\nrep V_{2}$. This shows (ii).

	(iii) If $j\not=k$, then, by (ii), Theorem \ref{thm:beligeneral}(iii) fails at $i=n+1$ for $M=M_{k}$ and $N=N_{j}$. Thus $N_{j}\nrep M_{k}$. If $j=k$, then $FN_{j}\rep FM_{k}$ by Lemma \ref{lem:P-spacerep}(iii), so $N_{j}\rep M_{k}$ by (ii), Lemma \ref{lem:Pomega-1}(i)(iii) and Theorem \ref{thm:beligeneral}. 
\end{proof}
\begin{lem}\label{lem:classicalminimal-even}
	Let $ n $ be an even integer and $ n\ge 2 $.
	\begin{enumerate}
		 \item[\rm (i)] If $ e>1 $, then $  P_{2}^{n+2}(\Delta) $ represents all lattices in $ \mathcal{C}_{e}^{n} $ except for $ H_{e}^{n}(1) $.
		 
		  \item[\rm (ii)] If $ e=1 $, then $ P_{1}^{n+2}(\omega) $ (resp. $P_{2}^{n+2}(\omega)$) represents all lattices in $ \mathcal{C}_{e}^{n} $ except for $  H_{e}^{n}(\Delta) $ (resp. $ H_{e}^{n}(1) $).
		 
		  \item[\rm (iii)] For any $c\in F^{\times}/F^{\times 2}$ with $d(c)\in \{0,1\}$, $  C_{1}^{n+2}(c) $ (resp. $  C_{2}^{n+2}(c) $) represents all lattices in $ \mathcal{C}_{e}^{n} $ except for $ C_{2}^{n}(c) $ (resp. $  C_{1}^{n}(c) $).
	\end{enumerate}
\end{lem}
\begin{proof}
	Take $N=C_{j}^{n}(c)$, with $j\in \{1,2\}$, $c\in F^{\times}/F^{\times 2}$ and $d(c)\in \{0,1\}$. 
	
	 (i)  By Lemma \ref{lem:P-spacerep}(i), we have $FN\rep  FP_{2}^{n+2}(\Delta) $. Combining Lemma \ref{lem:PDelta}(iii) and Theorem \ref{thm:beligeneral}, we futher have $N\rep P_{2}^{n+2}(\Delta)$.
		
	Since $ e>1 $, then $ H_{e}^{n}(\Delta) $ is ignored in $ \mathcal{C}_{e}^{n} $, it remains to show $H_{e}^{n}(1)\nrep  P_{2}^{n+2}(\Delta)$. This follows by Lemma \ref{lem:P-spacerep}(ii).
	
	(ii) Combining Lemmas \ref{lem:P-spacerep}(i), \ref{lem:Pomega-1}(ii) and Theorem \ref{thm:beligeneral}, we have $N\rep P_{k}^{n+2}(\omega)$ for $k=1,2$.
	
	 By Lemma \ref{lem:Pomega-2}(iii), we see that $H_{1}^{n}(1)\rep P_{1}^{n+2}(\omega)$, $H_{1}^{n}(\Delta)\nrep P_{1}^{n+2}(\omega)$, $H_{1}^{n}(\Delta)\rep P_{2}^{n+2}(\omega)$ and $H_{1}^{n}(1)\nrep P_{2}^{n+2}(\omega)$, as required.
	
	 (iii) Take $M=C_{1}^{n+2}(c)$, with $c\in F^{\times}/F^{\times 2}$ and $d(c)\in \{0,1\}$.  By Lemma \ref{lem:S-invariant}(iii), we have $R_{i}=0$ for $1\le i\le n+1$, $R_{n+2}=1-d(c)\in \{0,1\}$ and $\alpha_{n+1}=1$. If $R_{n+2}=0$, then $d(-a_{n+1}a_{n+2})=d(c)=1$. Since $d(-a_{i-1}a_{i})=\infty$ for $i\in [1,n]^{E}$, by the domination principle, we have $d((-1)^{(n+2)/2}a_{1,n+2})=d(-a_{n+1}a_{n+2})=1=1-R_{n+2}$. Hence $M$ satisfies the conditions in  Lemma \ref{lem:m=n+2representation}(i).
 
 	By \cite[Proposition 3.5(iii)]{HeHu2}, $ FC_{1}^{n+2}(c)\cong FW_{1}^{n+2}(c) $ represents $ FL$ for each lattice $ L $ in $ \mathcal{C}_{e}^{n} $ except for $ L\cong C_{2}^{n}(c) $. So (iii) follows by Lemma \ref{lem:m=n+2representation}(i). Similarly for $M=C_{2}^{n+2}(c)$.
\end{proof}
For odd $n\ge 3$ and $ c\in F^{\times}/F^{\times 2} $, we define the rank $n$ lattice
\begin{align*}
	\bar{C}_{1}^{n}(c):=\mathbf{H}_{0}^{(n-3)/2}\perp 
		\prec c,-c \omega,c\omega  \succ.
\end{align*}
\begin{lem}\label{lem:classicalminimal-odd}
	Let $ n $ be an odd integer and $ n\ge 3 $. Let $ \varepsilon\in\mathcal{O}_{F}^{\times} $.
	\begin{enumerate}
		\item[\rm (i)]  $ \bar{C}_{1}^{n+2}(\varepsilon) $ (resp. $C_{2}^{n+2}(\varepsilon) $ ) represents all lattices in $ \mathcal{C}_{e}^{n} $ except for $C_{2}^{n}(\varepsilon)  $ (resp. $C_{1}^{n}(\varepsilon)  $).
		
		\item[\rm (ii)] $ C_{1}^{n+2}(\varepsilon\pi)$ (resp. $ C_{2}^{n+2}(\varepsilon\pi)$) represents all lattices in $ \mathcal{C}_{e}^{n} $ except for $ C_{2}^{n}(\varepsilon\pi) $ (resp. $C_{1}^{n}(\varepsilon\pi) $). 
	\end{enumerate} 
\end{lem}
\begin{proof}
	(i) Take $M= \bar{C}_{1}^{n+2}(\varepsilon) $. By Lemma \ref{lem:BONGformaximallattice}, we have $R_{i}=0$ for $1\le i\le n+2 $. Since $R_{n+1}-R_{n}+d(-a_{n}a_{n+1})=d(\omega)=1 $, Proposition \ref{prop:alphaproperty}(vi) implies $\alpha_{n}=1$. Hence $ M $ satisfies the conditions in  Lemma \ref{lem:m=n+2representation}(ii). By \cite[Proposition 3.5(iii)]{HeHu2}, $ FM\cong FW_{1}^{n+2}(\varepsilon) $ represents $ FL$ for each lattice $ L $ in $ \mathcal{C}_{e}^{n} $ except for $ L\cong C_{2}^{n}(\varepsilon) $. So (i) follows by Lemma \ref{lem:m=n+2representation}(ii). A similar argument can be applied for  $M=C_{2}^{n+2}(\varepsilon)$.
	
	(ii) Take $M=C_{j}^{n+2}(\varepsilon\pi)$, $j\in \{1,2\}$. By Lemma \ref{lem:S-invariant}(ii), we have $R_{i}=0$ for $1\le i\le n+1$ and $R_{n+2}=1$. Hence $M $ satisfies the conditions in  Lemma \ref{lem:m=n+2representation}(ii). The remaining argument is similar to (i).
\end{proof}
\begin{proof}[Proof of Theorem \ref{thm:classicalnuniversaldyadic15theorem}]
	  (i) and (ii) follow by Lemma \ref{lem:classicalgivenM-15theorem} and Proposition \ref{prop:translation}. (iii) follows by Lemma \ref{lem:classicalminimal-even} for even $n\ge 2$ and by Lemma \ref{lem:classicalminimal-odd} for odd $n\ge 3$.
\end{proof}
 
\section{Applications to global representations} \label{sec:app}
Throughout this section, we assume that $K$ is an algebraic number field.
\begin{lem}\label{lem:defectP-p}
	Let $ E_{\mathfrak{P}}/K_{\mathfrak{p}} $ be a finite extension with discrete primes $\mathfrak{P}|\mathfrak{p}$ and $a\in  K_{\mathfrak{p}}$. 
	\begin{enumerate}
		\item[\rm (i)] We have $\ord_{\mathfrak{P}}(a)=\ord_{\mathfrak{p}}(a)e(\mathfrak{P}| \mathfrak{p})$;
		
		\item[\rm (ii)] We have $d_{\mathfrak{P}}(a)\ge d_{\mathfrak{p}}(a)e(\mathfrak{P}| \mathfrak{p})$.
		
		\item[\rm (iii)] If $\mathfrak{p}$ is dyadic and $L_{\mathfrak{p}}\cong \prec a_{1},\ldots, a_{m}\succ$ relative to a good BONG $\{x_{1},\ldots, x_{m}\}$, then $L_{\mathfrak{P}}\cong \prec a_{1},\ldots, a_{m}\succ$ relative to the good BONG $\{x_{1},\ldots, x_{m}\}$. 
	\end{enumerate} 
\end{lem}
\begin{proof}
	(i) See \cite[16:2]{omeara_quadratic_1963}. 
	
	(ii) Write $e=e(\mathfrak{P}|\mathfrak{p})$. For $a\in K_{\mathfrak{p}}$, if $\ord_{\mathfrak{p}}(a)$ is odd, then $d_{\mathfrak{P}}(a)\ge 0=d_{\mathfrak{p}}(a)e$. If $\ord_{\mathfrak{p}}(a)$ is even, we may assume   $a=1+\theta_{\mathfrak{p}}\pi_{\mathfrak{p}}^{k}$, with $\theta_{\mathfrak{p}}\in \mathcal{O}_{K_{\mathfrak{p}}}^{\times}$ and $k=d_{\mathfrak{p}}(a)$. Let $\Pi_{\mathfrak{P}} $ be a prime of $E_{\mathfrak{P}}$. Then $\pi_{\mathfrak{p}}=\varTheta_{\mathfrak{P}}\Pi_{\mathfrak{P}}^{e}$ for some $\varTheta_{\mathfrak{P}}\in \mathcal{O}_{E_{\mathfrak{P}}}^{\times}$, and so
	\begin{align*}
		a=1+\theta_{\mathfrak{p}}\pi_{\mathfrak{p}}^{k}=1+\theta_{\mathfrak{p}}\varTheta_{\mathfrak{P}}^{ek}\Pi_{\mathfrak{P}}^{ek}.
	\end{align*}
	Hence $d_{\mathfrak{P}}(a)\ge ke$ by the definition of quadratic defect.
	
	(iii) Write $r_{i}=\ord_{\mathfrak{p}}(a_{i})$ and $R_{i}=\ord_{\mathfrak{P}}(a_{i})$. Then $R_{i}=r_{i}e(\mathfrak{P}|\mathfrak{p})$ by (i) and $d_{\mathfrak{P}}(-a_{i}a_{i+1})\ge d_{\mathfrak{p}}(-a_{i}a_{i+1}))e(\mathfrak{P}|\mathfrak{p})$ by (ii). Hence, by \eqref{eq:BONGs}, we have
	\begin{align*}
		&R_{i+1}-R_{i}=(r_{i+1}-r_{i})e(\mathfrak{P}|\mathfrak{p})\ge -2e_{\mathfrak{p}}e(\mathfrak{P}|\mathfrak{p})=-2e_{\mathfrak{P}}, \\
		&	R_{i+1}-R_{i}+d_{\mathfrak{P}}(-a_{i}a_{i+1})\ge  (r_{i+1}-r_{i}+ d_{\mathfrak{p}}(-a_{i}a_{i+1}))e(\mathfrak{P}|\mathfrak{p})\ge 0\\
		\intertext{for $1\le i\le m-1$, and by \eqref{eq:GoodBONGs}, we also have}
		& R_{i+2}=r_{i+2}e(\mathfrak{P}|\mathfrak{p})\ge r_{i}e(\mathfrak{P}|\mathfrak{p})=R_{i}\quad\text{for $1\le i\le  m-2$.}
	\end{align*}
	Since $\{x_{1},\ldots,x_{m}\}$ is an orthogonal base of $K_{\mathfrak{p}}L_{\mathfrak{p}}$, it is also an orthogonal base of $E_{\mathfrak{P}}L_{\mathfrak{P}}$. Hence we are done by Lemma \ref{lem:goodBONGequivcon}.
\end{proof}
\begin{prop}\label{prop:Lp-nuniversal}
	Let $n$ be a positive integer. If an $ \mathcal{O}_{K} $-lattice $L$ represents all positive definite $\mathcal{O}_{K}$-lattices of rank $n$, then $L_{\mathfrak{p}}$ is $n$-universal for each non-archimedean prime $\mathfrak{p}$.  
	
	Thus if an $ \mathcal{O}_{K} $-lattice $L$ is $n$-universal, then $L_{\mathfrak{p}}$ is $n$-universal for each non-archimedean prime $\mathfrak{p}$. 
\end{prop}
\begin{proof}
	For a dyadic prime $\mathfrak{p}$, choose a positive definite classic integral $\mathcal{O}_{K}$-lattice $N$ of rank $n$. When $n\ge 2$, for each classic integral lattice $C$ in $\mathcal{C}_{e_{\mathfrak{p}}}^{n} $ (cf. Definition \ref{defn:classicallattices} and Proposition \ref{prop:translation}(iii)), by \cite[81:14]{omeara_quadratic_1963}, there exists an $\mathcal{O}_{K}$-lattice $M=M(\mathfrak{p},C)$ such that
	\begin{align*}
		\begin{cases}
			M_{\mathfrak{q}}=N_{\mathfrak{q}} &\text{if $\mathfrak{q}\not=\mathfrak{p}$},   \\
			M_{\mathfrak{p}}=C  &\text{if $\mathfrak{q}=\mathfrak{p}$}.
		\end{cases}
	\end{align*}
	By construction, $M$ is positive definite and classic integral. Hence, by the hypothesis, $L$ represents $M(\mathfrak{p},C)$ for each $C$ in $\mathcal{C}_{e_{\mathfrak{p}}}^{n}$. So $L_{\mathfrak{p}}$ represents all lattices in $\mathcal{C}_{e_{\mathfrak{p}}}^{n}$ and hence it is $n$-universal by Theorem \ref{thm:classicalnuniversaldyadic15theorem}. When $n=1$, put $\mathcal{C}_{e_{\mathfrak{p}}}^{1}=\{\langle \varepsilon\rangle,\langle \varepsilon\pi\rangle\mid \varepsilon\in\mathcal{O}_{F}^{\times}\}$ and use \cite[Lemma 2.2]{xu_indefinite_2020}. For a non-dyadic prime $\mathfrak{p}$, use the minimal testing sets in \cite[Proposition 3.2]{hhx_indefinite_2021} instead of $\mathcal{C}_{e_{\mathfrak{p}}}^{n}$. 
\end{proof}
\begin{proof}[Proof of Theorem \ref{thm:classic-n-universal-rankn+3}]
	First assume that $n\ge 2$ is even. Let $L\cong \langle h_{1},\ldots, h_{n+3}\rangle$ under the orthogonal base $\{x_{1},\ldots,x_{n+3}\}$ of $\mathbb{Q}L$ and $0\le \ord_{2}(h_{i})\le \ord_{2}(h_{i+1})$ for $1\le i\le n+2$. Since the discriminant of $K$ is even, there exists a dyadic prime $\mathfrak{p}$ such that $e_{\mathfrak{p}}=\ord_{\mathfrak{p}}(2)>1$. 
	
	By Lemma \ref{lem:goodBONGequivcon}, $L_{\mathfrak{p}}\cong \prec h_{1},\ldots,h_{n+3}\succ $ relative to the good BONG $\{(x_{1})_{\mathfrak{p}},\ldots,(x_{n+3})_{\mathfrak{p}}\}$. Assume that $L$ is $n$-universal over $\mathcal{O}_{K}$. By Proposition \ref{prop:Lp-nuniversal}, $L_{\mathfrak{p}}$ is $n$-universal, so, by Theorem \ref{thm:classicalnuniversaldyadic}, we have $	\ord_{\mathfrak{p}}(h_{i})=0$ for $1\le i\le n+1$ and $\ord_{\mathfrak{p}}(h_{n+2})\in \{0,1\}$. If $\ord_{\mathfrak{p}}(h_{n+2})=1$, then $ \ord_{2}(h_{n+2})e_{\mathfrak{p}}=1$ by Lemma \ref{lem:defectP-p}(i), which contradicts $e_{\mathfrak{p}}>1$. Hence $\ord_{\mathfrak{p}}(h_{n+2})=0$. For $1\le i\le n+2$, since $\ord_{\mathfrak{p}}(h_{i})=0$, we have $d_{\mathfrak{p}}(h_{i})=d_{\mathfrak{p}}(1+2k_{i})\ge e_{\mathfrak{p}}>1$ for some $k_{i}\in \mathbb{Z}$. Hence $d_{\mathfrak{p}}((-1)^{(n+2)/2}h_{1,n+2})>1$ by the domination principle, so $\ord_{\mathfrak{p}}(h_{n+3})\in \{0,1\}$ by Theorem \ref{thm:classicalnuniversaldyadic}(ii)(1)(a). Similarly, the case $\ord_{\mathfrak{p}}(h_{n+3})=1$ can be ruled out. Thus $\ord_{\mathfrak{p}}(h_{n+3})=0$ and hence $d_{\mathfrak{p}}(h_{n+3})>1$. It follows that $d_{\mathfrak{p}}(h_{i}h_{i+1})>1$ for $1\le i\le n+2$. This implies $e_{\mathfrak{p}}=1$ by Theorem \ref{thm:necessarycon}, a contradiction. A similar argument can be applied for odd $n\ge 1$.
\end{proof}
\begin{lem}\label{lem:classic-n-universal-not-overE}
		Let $ E_{\mathfrak{P}}/K_{\mathfrak{p}} $ be a finite extension with dyadic primes $\mathfrak{P}|\mathfrak{p}$ and $e(\mathfrak{P}|\mathfrak{p})>1$. If an $\mathcal{O}_{K_{\mathfrak{p}}}$-lattice $ L_{\mathfrak{p}} $ of rank $n+3$ is $n$-universal, then $L_{\mathfrak{P}}$ is not $n$-universal.
\end{lem}
\begin{proof}
	  Without loss of generality, assume that $n\ge 2$ is even.  Let $L_{\mathfrak{p}}\cong \prec a_{1},\ldots, a_{n+3}\succ$ relative to a good BONG $\{x_{1},\ldots, x_{n+3}\}$. Since $L_{\mathfrak{p}}$ is $n$-universal, $\{x_{1},\ldots,x_{n+3}\}$ is an orthogonal base of $L_{\mathfrak{p}}$ by Corollary \ref{cor:diagonalizable}. Hence $L_{\mathfrak{P}}\cong \prec a_{1},\ldots,a_{n+3}\succ$ relative to the good BONG $ \{x_{1},\ldots,x_{n+3}\}$ by Lemma \ref{lem:defectP-p}(iii).
	  
	 Write $r_{i}=\ord_{\mathfrak{p}}(a_{i})$ and $R_{i}=\ord_{\mathfrak{P}}(a_{i})$. Assume that $L_{\mathfrak{P}}$ is $n$-universal. Then, by Theorem \ref{thm:classicalnuniversaldyadic}, we have $	R_{i}=0$ for $1\le i\le n+1$ and $R_{n+2} \in \{0,1\}$. If $R_{n+2}=1$, then $ r_{n+2}e(\mathfrak{P}|\mathfrak{p}) =1$ by Lemma \ref{lem:defectP-p}(i), which contradicts $e(\mathfrak{P}|\mathfrak{p})>1$. Hence $R_{n+2}=0$. 
	
	From Lemma \ref{lem:defectP-p}(i) we have $e_{\mathfrak{P}}=e_{\mathfrak{p}}e(\mathfrak{P}|\mathfrak{p})>1$, so $d_{\mathfrak{P}}(-1)\ge e_{\mathfrak{P}}>1 $ by \cite[Lemma 2.1]{HeHu3}.
	
	For $1\le i\le n+2$, since $R_{i}=0$, Lemma \ref{lem:defectP-p}(i) implies $r_{i}=0$ and thus $d_{\mathfrak{p}}(a_{i})\ge 1$. Hence $d_{\mathfrak{P}}(a_{i})\ge d_{\mathfrak{p}}(a_{i})e(\mathfrak{P}|\mathfrak{p})>1$ by Lemma \ref{lem:defectP-p}(ii). By the domination principle, we have $d_{\mathfrak{P}}((-1)^{(n+2)/2}a_{1,n+2})>1$, so $R_{n+3}\in \{0,1\}$ by Theorem \ref{thm:classicalnuniversaldyadic}(ii)(1)(a). Similar to the argument for $R_{n+2}$, we have $R_{n+3}=0$ and thus $d_{\mathfrak{P}}(a_{n+3})>1$. By the domination principle, we have $d_{\mathfrak{P}}(-a_{i}a_{i+1})>1$ for $1\le i\le n+2$. 
	
	But the conditions of Theorem \ref{thm:classicalnuniversaldyadic}(ii)(1)(b) are satisified, hence $d_{\mathfrak{P}}(-a_{i}a_{i+1})=1-R_{i+1}\le 1$ for some $1\le j\le n+2$. A contradiction is derived. 
\end{proof}
\begin{proof}[Proof of Theorem \ref{thm:classic-n-universal-rankn+3-ramified}]
	 If $L$ is $n$-universal over $\mathcal{O}_{K}$, then $L_{\mathfrak{p}}$ is  $n$-univeral by Proposition \ref{prop:Lp-nuniversal}. Since $e(\mathfrak{P}|\mathfrak{p})>1$ from the hypothesis, $L_{\mathfrak{P}}$ is not $n$-universal by Lemma \ref{lem:classic-n-universal-not-overE}. Hence $L$ is not $n$-universal over $\mathcal{O}_{E}$ by Proposition \ref{prop:Lp-nuniversal}.
\end{proof}
\begin{proof}[Proof of Theorem \ref{thm:siegelthmII}]
	Let $\mathfrak{p}$ be a dyadic prime. Then $ (I_{m})_{\mathfrak{p}}\cong \prec 1,\ldots, 1\succ $ by Lemma \ref{lem:goodBONGequivcon}, and $ d_{\mathfrak{p}}(a_{i}(\mathfrak{p})a_{i+1}(\mathfrak{p}))=d_{\mathfrak{p}}(1)>1 $ for $1\le i\le m-1$. By the hypothesis and Proposition \ref{prop:Lp-nuniversal}, $(I_{m})_{\mathfrak{p}}$ is $n$-universal, so the necessity follows by Theorem \ref{thm:necessarycon}. 
	
	Conversely, let $N$ be a positive definite classic integral $\mathcal{O}_{K}$-lattice. Since $ K $ is not totally real, by strong approximation (cf. \cite[p.\hskip 0.1cm135]{hsia_indefinite_1998}), it is sufficient to show that $(I_{m})_{\mathfrak{p}}$ represents $N_{\mathfrak{p}}$ for each prime $\mathfrak{p}$. If $\mathfrak{p}$ is archimedean, then both  the dimension and the positive index of $K_{\mathfrak{p}}I_{m}$ are $m>n$. Hence $(I_{m})_{\mathfrak{p}}=K_{\mathfrak{p}}I_{m}$ is $n$-universal (as quadratic spaces) by \cite[Theorem 2.3]{hhx_indefinite_2021}. If $ \mathfrak{p} $ is non-dyadic, then $ (I_{m})_{\mathfrak{p}}\cong \langle 1,\ldots,1\rangle $ is $ n $-universal by \cite[Proposition 2.3]{xu_indefinite_2020} and \cite[Propositions 3.3, 3.4]{hhx_indefinite_2021}; if $ \mathfrak{p} $ is dyadic, then $ R_{i}((I_{m})_{\mathfrak{p}})=0 $ for $1\le i\le m$. For $n\ge 2$, $ (I_{m})_{\mathfrak{p}} $ is $ n $-universal by Theorem \ref{thm:classicalnuniversaldyadic}; for $n=1$, since $ e_{\mathfrak{p}}=1 $ by the hypothesis, Proposition \ref{prop:alphaproperty}(iii) implies that $ \alpha_{i}((I_{m})_{\mathfrak{p}})=1 $ for $1\le i\le m-1$. So $ (I_{m})_{\mathfrak{p}} $ is $ n $-universal by \cite[Theorem 2.1]{beli_universal_2020}. The proof is completed.
\end{proof}

\section*{Acknowledgments} 
The author would like to thank Prof. Yong Hu and Prof. Fei Xu for their helpful discussions and suggestions, and thank Prof. Beli for reading an earlier draft and providing many helpful comments. This work was supported by a grant from the National Natural Science Foundation of China (Project No. 12171223).

\end{document}